\documentclass[11pt, twoside, leqno]{article}

\usepackage{amssymb}
\usepackage{amsmath}
\usepackage{mathrsfs}
\usepackage{amsthm}
\usepackage{amsfonts}
\usepackage{color}
\usepackage{latexsym,amssymb}

\allowdisplaybreaks

\def\rn{{\mathbb{R}^n}}

\def\nn{{\mathbb N}}
\def\cg{{\mathcal G}}
\def\cs{{\mathcal S}}
\def\lap{\Delta}

\def\fz{\infty}
\def\az{\alpha}
\def\bz{\beta}

\def\ez{\epsilon}

\def\vi{\varphi}

\def\va{\varepsilon}

\def\lf{\left}
\def\r{\right}
\def\ls{\lesssim}
\def\noz{\nonumber}

\def\loc{{\mathop\mathrm{\,loc\,}}}
\def\supp{\mathop\mathrm{\,supp\,}}

\def\Xint#1{\mathchoice
{\XXint\displaystyle\textstyle{#1}}%
{\XXint\textstyle\scriptstyle{#1}}%
{\XXint\scriptstyle\scriptscriptstyle{#1}}%
{\XXint\scriptscriptstyle\scriptscriptstyle{#1}}%
\!\int}
\def\XXint#1#2#3{{\setbox0=\hbox{$#1{#2#3}{\int}$ }
\vcenter{\hbox{$#2#3$ }}\kern-.6\wd0}}

\def\dashint{\Xint-}

\def\f{\frac}

\def\t{{\theta}}
\def\lz{{\lambda}}
\def\d{{\delta}}

\def\al{{\alpha}}

\def\hs{\hspace{0.25cm}}

\def\CG{{\mathcal G}}
\def\CS{{\mathcal S}}
\def\hs{\hspace{0.25cm}}

\newcommand{\wt}{\widetilde}

\newtheorem{theorem}{Theorem}[section]
\newtheorem{lemma}[theorem]{Lemma}

\theoremstyle{definition}
\newtheorem{remark}[theorem]{Remark}

\renewcommand{\appendix}{\par
   \setcounter{section}{0}%
   \setcounter{subsection}{0}%
   \setcounter{subsubsection}{0}%
   \gdef\thesection{\@Alph\c@section}%
   \gdef\thesubsection{\@Alph\c@section.\@arabic\c@subsection}%
   \gdef\theHsection{\@Alph\c@section.}%
   \gdef\theHsubsection{\@Alph\c@section.\@arabic\c@subsection}%
   \csname appendixmore\endcsname
 }

\numberwithin{equation}{section}

\begin{document}

\arraycolsep=1pt

\title{\bf\Large Littlewood-Paley Characterizations of Fractional Sobolev Spaces via Averages on Balls
\footnotetext{\hspace{-0.35cm} 2010 {\it Mathematics Subject Classification}. Primary 46E35;
Secondary 42B25, 42B20, 42B35.
\endgraf {\it Key words and phrases.} Sobolev space, ball average, Lusin-area function, $g_\lambda^*$-function
\endgraf This project is supported by the National
Natural Science Foundation of China
(Grant Nos.~11571039 and 11471042),
the Specialized Research Fund for the Doctoral Program of Higher Education
of China (Grant No. 20120003110003) and the Fundamental Research
Funds for Central Universities of China (Grant No. 2014KJJCA10).}}
\author{Feng Dai, Jun Liu, Dachun Yang\,\footnote{Corresponding author}\ \
and Wen Yuan}
\date{}
\maketitle

\vspace{-0.8cm}

\begin{center}
\begin{minipage}{13cm}
{\small {\bf Abstract}\quad
In this paper, the authors characterize Sobolev spaces $W^{\alpha,p}({\mathbb R}^n)$
with the smoothness order $\alpha\in(0,2]$ and $p\in(\max\{1, \frac{2n}{2\alpha+n}\},\infty)$, via the Lusin area
function and the Littlewood-Paley $g_\lambda^\ast$-function in terms of centered ball averages.
The authors also show that the condition $p\in(\max\{1, \frac{2n}{2\alpha+n}\},\infty)$ is nearly sharp in the sense that these
characterizations are no longer true when $p\in (1,\max\{1, \frac{2n}{2\alpha+n}\})$.
These characterizations provide a new possible way to introduce fractional Sobolev spaces with smoothness order in $(1,2]$ on metric measure spaces.}
\end{minipage}
\end{center}

\section{Introduction\label{s1}}
\hskip\parindent The theory of Sobolev spaces is one of the central topics in  analysis on metric measure spaces. In the last two decades, several different approaches to introduce Sobolev spaces on
metric measure spaces were developed; see, for example, \cite{Haj, Shan, hkst, AMV, dgyy1, hyy15}.
In 1996, Haj\l asz \cite{Haj} introduced the notion
of Haj{\l}asz gradients, which has become an effective tool for developing the first order Sobolev spaces on metric
measure spaces. Soon after, Shanmugalingam \cite{Shan} introduced another kind
of the first order Sobolev spaces on metric measure spaces by means of the
notion of upper and weak upper gradients, which, comparing with Haj{\l}asz gradients, has stronger locality. Via introducing the fractional version of Haj{\l}asz gradients,
Hu \cite{hu03} and Yang \cite{y03} developed fractional Sobolev spaces
with smoothness order in $(0,1)$ on fractals and metric measure spaces, respectively.
Till now, the theory of Sobolev spaces with smoothness order in $(0,1]$
on metric measure spaces has been thoroughly
investigated and achieved great progresses (see, for example, \cite{fhk,hk00,h03}, the recent monograph \cite{hkst} and the references therein).

In recent years, there also exist some attempts to find a suitable way for
developing Sobolev spaces with smoothness order bigger
than $1$ on metric measure spaces. The key step is to find a suitable substitute of
high order derivatives on metric measure spaces.
In 2002, under a priori assumption on the existence of polynomials,
Liu et al. \cite{llw02} introduced high order Sobolev spaces on metric measure spaces.
Triebel \cite{t10,t11} and Haroske and Triebel
\cite{ht11,ht13} characterized Sobolev spaces
on $\rn$ with order bigger than $1$ via a pointwise
inequality involving high order differences, in spirit of Haj\l asz \cite{Haj}.
However, it remains unknown how to generalize
high order differences to metric measure spaces.

Let  $\lap:=\sum^n_{i=1}(\frac \partial{\partial x_i})^2$ be the Laplace operator.
Recall that, for all $p\in (1,\fz)$ and $\alpha\in(0,\fz)$,
the \emph {fractional Sobolev space $W^{\alpha,p}(\rn)$} is defined as the spaces consisting of all $f\in L^p(\rn)$ such that  $(-\lap)^{\alpha/2}f\in L^p(\rn)$, where $(-\lap)^{\alpha/2}f$ is defined via the Fourier
transform
by $[(-\lap)^{\alpha/2}f]^\wedge(\xi):=|\xi|^\al\widehat{f}(\xi)$ for all $\xi\in\rn$.
Recently, Alabern et al. \cite{AMV} proved that the Sobolev
spaces $W^{\alpha,p}(\rn)$ for $\alpha\in(0,2]$ and $p\in (1,\fz)$ can  be characterized by the square functions $\mathcal{\widetilde{S}}_\alpha(f)$ when $\az\in(0,2)$ and
 $\mathcal{\widetilde{S}}(f,g)$ when $\az=2$, which are defined, respectively, by setting, for all $f,\,g\in L^1_{{\rm loc}}(\rn)\cap \cs'(\rn)$,
 $$\mathcal{\widetilde{S}}_\alpha(f)(x):=\lf\{\int_0^\fz\lf|\dashint_{B(x,t)}
      \frac{f(y)-f(x)}{t^\alpha} \,dy\r|^2
      \,\frac{dt}t \r\}^{\frac 12},\quad  x\in\rn$$
 and
  $$\mathcal{\widetilde{S}}(f,g)(x):=\lf\{\int_0^\fz\lf|\dashint_{B(x,t)}
      \frac{f(y)-f(x)-B_tg(x)|y-x|^2}{t^2} \,dy\r|^2
      \,\frac{dt}t \r\}^{\frac 12},\quad x\in\rn.$$
Here and hereafter, $B(x,t)$ denotes the ball with center at $x\in \rn$ and radius  $t\in(0,\fz)$, and $\dashint_{B} g(y) \,dy$ denotes the \emph{integral average} of $g\in L^1_\loc(\rn)$ on a ball $B\subset \rn$, namely,
\begin{equation*}
\dashint_{B(x,t)} g(y) \,dy:= \frac1{|B(x,t)|}\int_{B(x,t)} g(y) \,dy=:B_tg(x).
\end{equation*}
More precisely, Alabern et al. in \cite[Theorems 1, 2 and 3]{AMV} proved the following Theorems \ref{Thm3} and \ref{Thm4}.

\setcounter{theorem}{0}
\renewcommand{\thetheorem}{\arabic{section}.\Alph{theorem}}

\begin{theorem}\label{Thm3}
Let $p\in(1,\fz)$ and $\al\in(0,2)$. Then the following statements are equivalent:
\begin{enumerate}
  \item[\rm{(i)}] $f\in W^{\al,p}(\rn)$;
  \item[\rm{(ii)}] $f\in L^p(\rn)$ and $\widetilde{\CS}_\al(f)\in L^p(\rn)$.
\end{enumerate}
In any of the above cases, $\|\widetilde{\CS}_\al(f)\|_{L^p(\rn)}$ is equivalent to $\|(-\lap)^{\al/2}f\|_{L^p(\rn)}$.
\end{theorem}

\begin{theorem}\label{Thm4}
Let $p\in (1,\fz)$. Then the following statements are equivalent:
\begin{enumerate}
\item[\rm{(i)}] $f\in W^{2,p}(\rn)$;
\item[\rm{(ii)}] $f\in L^p(\rn)$ and there exists a function $g\in L^p(\rn)$ such that $\widetilde{S}(f,g)\in L^p(\rn)$.
\end{enumerate}

If $f\in W^{2,p}(\rn)$, then one can take $g:=\lap f/(2n)$ and, if (ii) holds true,
then necessarily $g:=\lap f/(2n)$ almost everywhere.
In any of the above cases, $\|\widetilde{\mathcal{S}}(f,g)\|_{L^p(\rn)}$ is equivalent to $\|\lap f\|_{L^p(\rn)}$.
\end{theorem}

The proofs of Theorems \ref{Thm3} and \ref{Thm4} in \cite{AMV} are based on
the theory of vector valued singular integrals (see, for example,
\cite[Theorem 3.4]{G-CF}).
Recently, Haj\l asz and Liu \cite{hl} provided a new and simplified proof of
Theorem \ref{Thm3} when $\az=1$, and
also established a Marcinkiewicz integral type characterization
of $W^{1,p}(\rn)$ (see \cite[Theorem 1.4]{hl}), which states that, for all $p\in(1,\fz)$, $f\in W^{1,p}(\rn)$
if and only if $f\in L^p(\rn)$ and
$$\lf\{\int_0^\fz\lf| f(\cdot)-\frac1{|S(\cdot,t)|}\int_{S(\cdot,t)}f(y) \,d\sigma(y)\r|^2
      \,\frac{dt}t \r\}^{\frac 12}\in L^p(\rn),$$
where $S(x,t)$ denotes the sphere in $\rn$ with center $x\in\rn$ and radius $t\in (0,\fz)$,
and $d\sigma$ the Lebesgue surface measure on $S(x,t)$. However, it is unclear how
to define sphere integrals in a general metric measure setting.
On the other hand, Sato in \cite[Theorem 1.5]{s14} proved that, for any
$\az\in(0,2)$, $p\in(1,\fz)$, Muckenhoupt $A_p$-weight $w$ and Schwartz function $f$,
the weighted Lebesgue norm $\|f\|_{L^p_w(\rn)}$
is equivalent to $\|T_\az(f)\|_{L^p_w(\rn)}$, with equivalent
positive constants independent of $f$, where 
$$T_\az(f)(x):=\lf\{\int_0^\fz\,\lf|I_\az(f)(x)-\Phi_t\ast I_\az(f)(x)\r|^2\frac{dt}{t^{1+2\az}}\r\}^{1/2},\quad x\in\rn,$$
$I_\az$ denotes the Riesz potential operator defined by
$\widehat{I_\az(f)}(\xi):=(2\pi|\xi|)^{-\az} \widehat{f}(\xi)$
for all $\xi\in\rn\setminus\{0\}$,   $\Phi$ is a bounded radial function on
$\rn$ with compact support satisfying $\int_\rn \Phi(x)\,dx=1$,
and $\Phi_t(x):=t^{-n}\Phi(x/t)$ for all $x\in\rn$ and $t\in(0,\fz)$.
This, via taking $\Phi:=|B(0,1)|^{-1}\chi_{B(0,1)}$,
further induces a weighted version of Theorem \ref{Thm3}, 
that is, $f\in W^{\az,p}_w(\rn)$
if and only if both $f$ and $\widetilde{S}_\az(f)$ belong to the weighted Lebesgue space $L^p_w(\rn)$ (see \cite[Corollary 1.2]{s14}).

We also point out that
the corresponding characterization of Theorems \ref{Thm3} and \ref{Thm4}
for Sobolev spaces with smoothness order bigger
than $2$ was also obtained in  \cite{AMV}. But, in this article, for the presentation simplicity,
we only consider Sobolev spaces with smoothness order in $(0,2]$.
These characterizations provide a possible way to introduce the
Sobolev spaces with the smoothness order bigger than 1 on metric
measure spaces.

Notice that $\widetilde{\mathcal{S}}_\alpha(f)$ and $\mathcal{\widetilde{S}}(f,g)$  above can be reformulated, respectively, as follows:
\begin{equation}\label{g-fn}
\cg_\al(f)(x) := \lf\{\int_0^\fz \lf|\frac{B_tf(x)-f(x)}{t^\alpha}\r|^2\,\frac{dt}t\r\}^{\frac12},\quad \al\in(0,2)\ {\rm and}\ x\in\rn
\end{equation}
and
\begin{equation}\label{g-fn1}
\cg(f,g)(x) := \lf\{\int_0^\fz \lf|\frac{B_tf(x)-f(x)}{t^2}-B_tg(x)\r|^2\,\frac{dt}t\r\}^{\frac12},\quad x\in\rn,
\end{equation}
which can be seen as the \emph{Littlewood-Paley $g$-function} of
$\frac{B_t f-f}{t^\al}$ and $\frac{B_t f-f}{t^2}-B_tg$, respectively.
Therefore, it is a \emph{natural question} to ask whether the corresponding
Lusin area function and the corresponding Littlewood-Paley $g_\lz^\ast$-function can characterize
$W^{\alpha,p}(\rn)$ or not. Here the Lusin area function and the Littlewood-Paley $g_\lz^\ast$-function
are defined, respectively, by setting, for any $f,\,g\in L^1_{\loc}(\rn)$, $\lz\in(1,\fz)$ and $x\in\rn$,
\begin{equation}\label{s-fn}
\cs_\al(f)(x):= \lf\{\int_0^\fz \int_{B(x,t)}\lf|\frac{B_tf(y)-f(y)}{t^\alpha}\r|^2\, dy\,
\frac{dt}{t^{n+1}}\r\}^{\frac12}
\end{equation}
and
\begin{equation}\label{gl-fn1}
\cg_{\al,\lz}^\ast(f)(x) := \lf\{\int_0^\fz \int_\rn
  \lf|\frac{B_tf(y)-f(y)}{t^\alpha}\r|^2
  \lf(\frac t {t+|x-y|}\r)^{\lz n} \,dy\,
  \frac{dt}{t^{n+1}}\r\}^{\frac12}
\end{equation}
for $\alpha\in(0,2)$, and
\begin{equation}\label{s-fn1}
\cs(f,g)(x):= \lf\{\int_0^\fz \int_{B(x,t)}\lf|\frac{B_tf(y)-f(y)}{t^2}-B_tg(y)\r|^2\, dy\,
\frac{dt}{t^{n+1}}\r\}^{\frac12}
\end{equation}
and
\begin{equation}\label{gl-fn}
\cg_{\lz}^\ast(f,g)(x) := \lf\{\int_0^\fz \int_\rn
  \lf|\frac{B_tf(y)-f(y)}{t^2}-B_tg(y)\r|^2
  \lf(\frac t {t+|x-y|}\r)^{\lz n} \,dy\,
  \frac{dt}{t^{n+1}}\r\}^{\frac12}.
\end{equation}
To answer this question, He et al. \cite{hyy15} established the
following characterizations of
the second order Sobolev spaces.

\begin{theorem}\label{Thm0}
\begin{itemize}
\item[{\rm(I)}] If $p\in[2,\fz)$ and $n\in\nn:=\{1,2,\ldots\}$, then the following statements are equivalent:
\begin{enumerate}
  \item[\rm{(i)}] $f\in W^{2,p}(\rn)$;
  \item[\rm{(ii)}] $f\in L^p(\rn)$ and there exists $g\in L^p(\rn)$ such that $\CS(f,g)\in L^p(\rn)$;
  \item[\rm{(iii)}] $f\in L^p(\rn)$ and there exists $g\in L^p(\rn)$ such that $\CG_\lz^\ast(f,g)\in L^p(\rn)$ for some  $\lz\in(1,\fz)$.
\end{enumerate}
\item[{\rm(II)}] If $p\in(1,2)$ and $n\in\{1,2,3\}$, then the following statements are equivalent:
\begin{enumerate}
  \item[\rm{(i)}] $f\in W^{2,p}(\rn)$;
  \item[\rm{(ii)}] $f\in L^p(\rn)$ and there exists $g\in L^p(\rn)$ such that $\CS(f,g)\in L^p(\rn)$;
  \item[\rm{(iii)}] $f\in L^p(\rn)$ and there exists $g\in L^p(\rn)$ such that $\CG_\lz^\ast(f,g)\in L^p(\rn)$ for some  $\lz\in(2/p,\fz)$.
\end{enumerate}
\end{itemize}

Moreover, if $f\in W^{2,p}(\rn)$,
then $g$ in (ii) and (iii) of (I) and (II) can be taken as $g:=\frac{\lap f}{2n+4}$; while
if either of (ii) and (iii) in (I) and (II) holds true, then $g:=\frac{\lap f}{2n+4}$ almost everywhere.
In any of above cases, $\|\CS(f,g)\|_{L^p(\rn)}$ and $\lf\|\CG_\lz^*(f,g)\r\|_{L^p(\rn)}$
are equivalent to $\|\lap f\|_{L^p(\rn)}$.
\end{theorem}

\setcounter{theorem}{0}
\renewcommand{\thetheorem}{\arabic{section}.\arabic{theorem}}

Observe that, when $p\in[2,\fz)$, Theorem \ref{Thm0} gives a complete answer to the previous question
for $W^{2,p}(\rn)$.  However, when $p\in(1,2)$, the question was answered only for $W^{2,p}(\rn)$
when $n\in\{1,2,3\}$.

The main purpose of this paper is to give a more complete answer to the previous question
on characterizations of fractional Sobolev spaces  with smoothness order
in $(0,2]$ via the Lusin-area or the Littlewood-Paley $g_\lz^*$-functions.
We first have the following characterizations of $W^{2,p}(\rn)$, which complement
Theorem \ref{Thm0} when $p\in(1,2)$.

\begin{theorem}\label{Thm2}
Let $n\in[4,\fz)\cap\nn$.

\begin{itemize}
\item[{\rm(I)}] If $p\in(\f{2n}{4+n},2)$, then the following statements are equivalent:
\begin{enumerate}
  \item[\rm{(i)}] $f\in W^{2,p}(\rn)$;
  \item[\rm{(ii)}] $f\in L^p(\rn)$ and there exists $g\in L^p(\rn)$ such that $\CS(f,g)\in L^p(\rn)$;
  \item[\rm{(iii)}] $f\in L^p(\rn)$ and there exists $g\in L^p(\rn)$ such that $\CG_\lz^\ast(f,g)\in L^p(\rn)$ for some  $\lz\in(2/p,\fz)$.
\end{enumerate}

Moreover, if $f\in W^{2,p}(\rn)$,
then $g$ in (ii) and (iii) can be taken as $g:=\frac{\lap f}{2n+4}$; while
if either of (ii) and (iii) holds true, then $g:=\frac{\lap f}{2n+4}$ almost everywhere.
In any of above cases, $\|\CS(f,g)\|_{L^p(\rn)}$ and $\lf\|\CG_\lz^*(f,g)\r\|_{L^p(\rn)}$
are equivalent to $\|\lap f\|_{L^p(\rn)}$.

\item[{\rm(II)}] If $p\in(1,\f{2n}{4+n})$,
then the equivalence between (i) and
either (ii) or (iii) in (I) no longer  holds true.
\end{itemize}
\end{theorem}

We remark that (ii) and (iii) in Theorem \ref{Thm2}(I) are equivalent for all $p\in(1,\fz)$ and $n\in\mathbb{N}$ (see Lemma \ref{e34x} or the proof of Theorem \ref{Thm2} below). Moreover, from Theorem \ref{Thm2}(II), we deduce that the condition $p\in (\frac{2n}{4+n},\fz)$ in Theorem \ref{Thm2}(I)
is nearly sharp in the sense that, if (i) in Theorem \ref{Thm2}(I) is equivalent to either (ii) or (iii) of Theorem \ref{Thm2}(I), then one must have $p\in [\frac{2n}{4+n},\fz)$.
This also implies that the equivalence between (i) in Theorems \ref{Thm0}(I) and \ref{Thm2}(II) and either
of (ii) and (iii) holds true for all $p\in(1,\fz)$ if and only if $n\leq4$. We also point out that,
in Theorem \ref{Thm2}, the case when $p=\f{2n}{4+n}$ is still unclear.

As the fractional variant of Theorem \ref{Thm0} and Theorem \ref{Thm2}, we have the
following characterizations for $W^{\az,p}(\rn)$ with $\az\in(0,2)$ and $p\in(1,\fz)$.

\begin{theorem}\label{Thm1}
Let $n\in\nn$ and $\alpha\in(0,2)$.

\begin{itemize}
\item[{\rm(I)}] If $p\in(\max\{1, \f{2n}{2\al+n}\},\fz)$, then the following statements are equivalent:
\begin{enumerate}
  \item[\rm{(i)}] $f\in W^{\alpha,p}(\rn)$;
  \item[\rm{(ii)}]  $f\in L^p(\rn)$ and $\CS_\al(f)\in L^p(\rn)$;
  \item[\rm{(iii)}] $f\in L^p(\rn)$ and $\CG_{\al,\lz}^\ast(f)\in L^p(\rn)$ for some  $\lz\in(\max\{1,2/p\},\fz)$.
\end{enumerate}

Moreover, $\|\CS_\al(f)\|_{L^p(\rn)}$ and $\|\CG_{\al,\lz}^*(f)\|_{L^p(\rn)}$
are equivalent to $\|(-\lap)^{\al/2}f\|_{L^p(\rn)}$.

\item[\rm(II)] If $n>2\al$ and $p\in(1,\f{2n}{2\al+n})$, then the equivalence between (i) and
either of (ii) or (iii) in (I) no longer holds true.
\end{itemize}
\end{theorem}

We remark that (ii) and (iii) of Theorem \ref{Thm1}(I) are equivalent
for all $p\in(1,\fz)$ and $n\in\mathbb{N}$ (see Lemma \ref{e34x} or
the proof of Theorem \ref{Thm1} below). Moreover, the condition
$p\in(\max\{1, \f{2n}{2\al+n}\},\fz)$ is also nearly sharp
in the sense that, if (i) of Theorem \ref{Thm1}(I) is equivalent to either (ii) or (iii) of
Theorem \ref{Thm1}(I), then one has
$$p\in \lf[\frac{2n}{2\al+n},\fz\r)\ {\rm for}\ n>2\al\ \quad
{\rm and}\quad \ p\in (1,\fz)\ {\rm for}\ n\leq2\al;$$ see Theorem \ref{Thm1}(II).
We also point out that, in Theorem \ref{Thm1}, the case when $n>2\al$ and $p=\f{2n}{2\al+n}$
is still unclear.

The proofs of Theorems \ref{Thm2} and  \ref{Thm1} are respectively
presented in Section \ref{s3} and Section  \ref{s2}.
Comparing with the proofs of \cite[Theorems 1, 2 and 3]{AMV} and \cite[Theorem 1.1]{hyy15} (Theorem \ref{Thm0}),
the proofs of Theorems \ref{Thm2} and  \ref{Thm1} are more subtle and complicated.
The main new idea is to control the Lusin area functions
$\CS(f,g)$ and $\CS_\al(f)$ by a sum of a sequence of  convolution operators
whose kernels satisfy vector-valued H\"ormander conditions.
Then, by applying the vector-valued Calder\'on-Zygmund theory
(see \cite[Theorem 3.4]{G-CF}) and
 the Marcinkiewicz interpolation theorem (see \cite[Theorem 1.3.2]{lg08}),
we obtain the boundedness of all such   convolution operators
on $L^p(\rn)$ as well as the exact decay estimates of their operator norms,
which implies the desired boundedness of Lusin area functions.
On the other hand,
in the proofs of \cite[Theorems 1, 2 and 3]{AMV}, the norm estimates
$\|(-\lap)^{\al/2}f\|_{L^p(\rn)}\ls \|\CG_\al(f)\|_{L^p(\rn)}$ and $\|\lap f\|_{L^p(\rn)}\ls \|\CG(f,g)\|_{L^p(\rn)}$ were obtained via the polarization
and a duality argument, which is not feasible for us to obtain
the norm estimates that $\|(-\lap)^{\al/2}f\|_{L^p(\rn)}\ls \|\CS_\al(f)\|_{L^p(\rn)}$,
$\|(-\lap)^{\al/2}f\|_{L^p(\rn)}\ls \|\CG^*_{\al,\lz}(f)\|_{L^p(\rn)}$
and  $\|\lap f\|_{L^p(\rn)}\ls \|\CS(f,g)\|_{L^p(\rn)}$,
$\|\lap f\|_{L^p(\rn)}\ls \|\CG^*_{\lz}(f,g)\|_{L^p(\rn)}$ in Theorems \ref{Thm2} and  \ref{Thm1}, respectively.
To overcome this difficulty, we make use of the fact
$\dot W^{\alpha,p}(\rn)=\dot F^\alpha_{p,2}(\rn)$ and prove that $\|f\|_{\dot F^\alpha_{p,2}(\rn)}\ls \|\CS_{\al}(f)\|_{L^p(\rn)}$ with $\alpha\in(0,2)$ and $\|f\|_{\dot F^2_{p,2}(\rn)}\ls \|\CS(f,g)\|_{L^p(\rn)}$
by means of the classical Lusin area function characterization of the Triebel-Lizorkin space
$\dot F^\alpha_{p,2}(\rn)$ (see, for example, \cite{u12})
 and the Fefferman-Stein vector-valued inequality from \cite{fs}.

Notice that the characterizations of fractional Sobolev spaces in (ii) and (iii)
of Theorems \ref{Thm1}(I) and \ref{Thm2}(I) are independent of the differential structure of
$\rn$. In this sense, these characterizations provide a new possible way to
introduce fractional Sobolev spaces with smoothness order in $(1,2]$ on metric measure
spaces.

Finally, we make some conventions on notation. Denote by
$L^1_\loc(\rn)$ the collection of all
locally integrable functions  on $\rn$.
Let  $\mathcal{S}(\rn)$ denote the collection of all
\emph{Schwartz functions} on $\rn$, endowed
with the usual topology, and $\mathcal{S}'(\rn)$ its \emph{topological dual}, namely,
the collection of all bounded linear functionals on $\mathcal{S}(\rn)$
endowed with the weak $\ast$-topology.
Let $\mathbb{Z}_+:=\mathbb{N}\cup\{0\}$ and,
for all $\az\in\mathbb{Z}_+^n$, $m\in\mathbb{Z}_+$ and $\vi\in\cs(\rn)$,  let
$$\|\vi\|_{\az,m}:=\sup_{|\bz|\le |\az|,\,x\in\rn}(1+|x|)^m |\partial^\bz \vi(x)|.$$

The symbol $\widehat{\varphi}$ refers to  the \emph{Fourier transform},
 and $\varphi^{\vee}$ to its \emph{inverse transform},
both defined on $\cs'(\rn)$.
For any $\vi\in \cs(\rn)$ and $t\in(0,\fz)$, we let $\vi_t(\cdot):=t^{-n}\vi(\cdot/t)$.
For any $E\subset\rn$, let $\chi_E$ be its \emph{characteristic function}.

The symbol $C$  denotes a {\it positive
constant}, which is independent of the main parameters involved but
whose value may differ from line to line, and the symbol $C_{(\al,\,\bz,\ldots)}$ denotes
a positive constant depending on the parameters $\al,\,\bz,\ldots$.
We use the symbol $f\ls g$ to denote that there exists a positive constant $C$ such that  $f\le
C g$ and, if $f\ls g\ls f$, we then write $f\sim g$.

\section{Proof of Theorem \ref{Thm2}\label{s3}}

\hskip\parindent
To prove Theorem \ref{Thm2},
we first consider the relations among the Littlewood-Paley $g$-functions in \eqref{g-fn} and \eqref{g-fn1},
the Lusin area functions in \eqref{s-fn} and \eqref{s-fn1}  and the Littlewood-Paley $g_\lz^\ast$-functions in \eqref{gl-fn1} and \eqref{gl-fn}, respectively.
To this end, for all $\lz\in(1,\fz)$, measurable functions $\mathcal{F}$ on $\rn\times (0,\fz)$ and $x\in\rn$, define
\begin{eqnarray*}
  \mathcal{G}(\mathcal{F})(x) &:=& \lf\{\int_0^\fz |\mathcal{F}(x,t)|^2\,\frac{dt}t\r\}^{\frac12}, \nonumber\\
  \mathcal{S}(\mathcal{F})(x) &:=& \lf\{\int_0^\fz \int_{B(x,t)} |\mathcal{F}(y,t)|^2 \,dy\, \frac{dt}{t^{n+1}}\r\}^{\frac12}\nonumber
\end{eqnarray*}
and
$$ \mathcal{G}_{\lz}^\ast(\mathcal{F})(x) := \lf\{\int_0^\fz \int_\rn |\mathcal{F}(y,t)|^2
  \lf(\frac t {t+|x-y|}\r)^{\lz n} \,dy\,
  \frac{dt}{t^{n+1}}\r\}^{\frac12}.$$

Concerning the relationships among the functions $\mathcal{G}(\mathcal{F})$, $\mathcal{S}(\mathcal{F})$
 and $\mathcal{G}_{\lz}^\ast(\mathcal{F})$, we have the following two lemmas, which
 come from \cite[Lemmas 2.1(iii) and 2.2]{hyy15}, respectively.

\begin{lemma}\label{lem-2-3}
Let $\lz\in(1,\fz)$ and $p\in[2,\fz)$.
Then, for all measurable functions
$\mathcal{F}$ on $\rn\times (0,\fz)$,
$\|\mathcal{G}_\lz^\ast(\mathcal{F})\|_{L^p(\rn)}\le C\|\mathcal{G}(\mathcal{F})\|_{L^p(\rn)}$, where $C$ is a positive constant independent of $\mathcal{F}$.
\end{lemma}

\begin{lemma}\label{e34x}
Let $p\in(1,\fz)$ and $\lz\in(\max\{2/p,1\},\fz)$. Then, for all
measurable functions $\mathcal{F}$ on $\rn\times (0,\fz)$,
$\mathcal{S}(\mathcal{F})\in L^p(\rn)$ if and only if $\mathcal{G}_\lz^\ast(\mathcal{F})\in L^p(\rn)$. Moreover,
the $ L^p(\rn)$-norm of $\mathcal{G}_\lz^\ast(\mathcal{F})$ is equivalent to that of $\mathcal{S}(\mathcal{F})$
with the equivalent positive constants independent of $\mathcal{F}$.
\end{lemma}

Applying Lemma \ref{lem-2-3}, we obtain the following conclusions.

\begin{lemma}\label{Lem p>2}
Let $p\in[2,\fz)$ and $\lz\in(1,\fz)$.

{\rm (i)} If $\al\in(0,2)$, then there exist positive constants $C_1$ and $C_2$ such that,
for all $f\in W^{\al,p}(\rn)$, $$\|\mathcal{S}_\al(f)\|_{L^p(\rn)}\le C_1\|\CG_{\al,\lz}^\ast(f)\|_{L^p(\rn)}
\le C_2\|(-\lap)^{\al/2}f\|_{L^p(\rn)}.$$ In particular, for all $f\in W^{\al,2}(\rn)$,
$\|\CS_\al(f)\|_{L^2(\rn)}$ and $\|\CG_{\al,\lz}^\ast(f)\|_{L^2(\rn)}$ are both equivalent to
$\|(-\lap)^{\al/2}f\|_{L^2(\rn)}$ with the equivalent positive constants independent of $f$.

{\rm (ii)} There exist positive constants $C_3$ and $C_4$ such that,
for all $f\in W^{2,p}(\rn)$, $$\|\mathcal{S}(f,g)\|_{L^p(\rn)}\le C_3\|\CG_{\lz}^\ast(f,g)\|_{L^p(\rn)}
\le C_4\|\lap f\|_{L^p(\rn)},$$ where $g:=\f{\lap f}{2n+4}$. In particular, for all $f\in W^{2,2}(\rn)$,
$\|\CS(f,g)\|_{L^2(\rn)}$ and $\|\CG_{\lz}^\ast(f,g)\|_{L^2(\rn)}$ are both equivalent to
$\|\lap f\|_{L^2(\rn)}$ with the equivalent positive constants independent of $f$.
\end{lemma}

\begin{proof} We first prove (i).
Let $f\in W^{\al,p}(\rn)$. Applying Lemmas \ref{e34x} and \ref{lem-2-3} with
\begin{equation*}
\mathcal{F}:=\mathcal{F}_\al(x,t):=\lf|\frac{B_tf(x)-f(x)}{t^\al}\r|,\quad (x,t)\in\rn\times (0,\fz)
\end{equation*}
and \cite[Theorems 1 and 3]{AMV}, we see that
\begin{eqnarray*}
\lf\|\CS_\al(f)\r\|_{L^p(\rn)}&&=\lf\|\mathcal{S}(\mathcal{F})\r\|_{L^p(\rn)}
\ls\lf\|\CG_{\lz}^\ast(\mathcal{F})\r\|_{L^p(\rn)}
\\
&&\sim\lf\|\CG_{\al,\lz}^\ast(f)\r\|_{L^p(\rn)}\lesssim \lf\|\CG_\al(f)\r\|_{L^p(\rn)}\sim
\lf\|(-\lap)^{\al/2}f\r\|_{L^p(\rn)}.
\end{eqnarray*}

In particular, if $p=2$, then, for all $f\in W^{\al,2}(\rn)$, by the Fubini theorem and \cite[(16)]{AMV},
we have
\begin{eqnarray*}
  \|\CS_\al(f)\|_{L^2(\rn)}^2 &=&\int_\rn\int_0^\fz \int_{B(x,t)}\lf|\frac{B_tf(y)-f(y)}{t^\al}\r|^2 \,dy\,\frac{dt}{t^{n+1}}\,dx\\
   &=&\int_\rn\int_0^\fz\int_{B(y,t)}\lf|\frac{B_tf(y)-f(y)}{t^\al}\r|^2 \,dx \,\frac{dt}{t^{n+1}}\, dy\\
   &=&\widetilde{c}\|\CG_\al(f)\|_{L^2(\rn)}^2
   =c\lf\|(-\lap)^{\al/2}f\r\|_{L^2(\rn)}^2,
\end{eqnarray*}
where $\widetilde{c}$ and $c$ are positive constants depending only on $n$. This proves (i).

Now we prove (ii). This time, by the same argument as that used in (i) with
\begin{equation*}
\mathcal{F}:=\mathcal{F}(x,t):=\lf|\frac{B_tf(x)-f(x)}{t^2}-B_tg(x)\r|,\quad (x,t)\in\rn\times (0,\fz),
\end{equation*}
we obtain the desired conclusion of (ii). This finishes the proof of Lemma \ref{Lem p>2}.
\end{proof}
 For all $x\in\rn$, let
 \begin{align*}
    \mathcal{K}(x) :=\chi\ast I_2(x) -I_2(x)+\f1{2n+4}\chi(x),
 \end{align*}
 where $I_2$ is a distribution  defined by $\widehat{I_2}(\cdot):=|\cdot|^{-2}$
 and $\chi:=\f 1 {|B(0,1)|} \chi_{B(0,1)}$. Then
\begin{equation}\label{1-4'}\widehat{\mathcal{K}}(t\xi) = \f{\widehat{\chi}(t\xi)-1}{t^2|\xi|^{2}}+\f{\widehat{\chi}(t\xi)}{2n+4}
 \ \ {\rm and}\ \
 f\ast \mathcal{K}_t (x) = \f {(B_t -I)(f\ast I_2)(x)}{t^2}+\f{B_tf(x)}{2n+4}\end{equation}
 for all $t\in(0,\fz)$ and $\xi,\,x\in\rn$, where $I$ denotes the identity operator.

We have the following conclusions.

\begin{lemma}\label{lem0'}
{\rm(i)}
There exists a positive constant $C$ such that, for all $\xi\in\rn$,
\begin{equation*}
\lf|\widehat{\chi}(\xi)\r|\leq C
\lf(\f 1{1+|\xi|}\r)^{\f {n+1}2}.
\end{equation*}
{\rm(ii)}
  $\f {\widehat{\chi}(\xi) -1}{|\xi|^2}$ is a radial $C^\infty$ function on $\rn$ satisfying that
  \begin{equation*}\sup_{\xi\in\rn}
  \lf| \nabla^i \lf(  \f { \widehat{\chi}(\xi) -1}{|\xi|^2}\r)\r| \leq C <\infty,\   \   \  i\in\mathbb{Z}_+.\end{equation*}
{\rm(iii)}
There exists a positive constant $C$
such that, for all $\xi\in\rn$ with $|\xi|\in[0,1]$,
 \begin{equation*}
 \lf|\f{\widehat{\chi}(\xi)-1}{|\xi|^{2}} +\f1{2n+4}\r|\leq C|\xi|^2.
 \end{equation*}
{\rm(iv)}
Let $n\in[3,\fz)\cap\mathbb{N}$. Then there exists a positive constant $C$
such that, for all $\xi\in\rn$,
 $$|\widehat{\mathcal{K}}(\xi)|\leq C \min\{ |\xi|^{2}, |\xi|^{-2}\}.$$
\end{lemma}

\begin{proof}
The item (i) can be found in \cite[Appendix B.5, pp. 429-431]{lg08} and (ii) is from \cite[Lemma 2.1]{dgyy2}.

Now we prove (iii). Since $\chi$ is radial, we know that $\widehat{\chi}$ is also radial. Hence, for all
$\xi\in\rn$,
\begin{eqnarray*}
\widehat{\chi}(\xi)=\widehat{\chi}(|\xi|e_1)=\dashint_{B(0,1)}e^{-ix_1|\xi|}\,dx,
\end{eqnarray*}
where $e_1:=(1,0,\ldots,0)\in\rn$.
Therefore, if $|\xi|\leq1$, then, by the Taylor expansion, we see that
\begin{eqnarray*}
&&\lf|\f{\widehat{\chi}(\xi)-1}{|\xi|^{2}} +\f1{2n+4}\r|\\
&&\quad=\lf|-\f12\dashint_{B(0,1)}x_1^2\,dx+\f{|\xi|^2}{4!}\dashint_{B(0,1)}x_1^4\,dx
+o(|\xi|^2)\dashint_{B(0,1)}x_1^6\,dx+\f1{2n+4}\r|
\ls|\xi|^2.
\end{eqnarray*}
This proves (iii).

Finally we prove (iv). If $|\xi|\leq1$, then, by \eqref{1-4'}, (ii) and (iii), we have
  $$\lf|\widehat{\mathcal{K}}(\xi)\r| \leq\lf|\f{\widehat{\chi}(\xi)-1}{|\xi|^{2}} +\f1{2n+4}\r|+\lf|\f1{2n+4}\lf[1-\widehat{\chi}(\xi)\r]\r|\ls|\xi|^2.$$
  If $|\xi|\geq1$, then the desired estimate follows from \eqref{1-4'} and (i) with $n\geq3$.
  This finishes the proof of Lemma \ref{lem0'}.
\end{proof}

\begin{lemma}\label{Lem'}
Let $p\in(\max\{1, \f{2n}{4+n}\},\fz)$. Then there exists a positive constant $C$
such that, for all $f\in L^p(\rn)$,
 \begin{equation}\label{2-1-1'}\|\wt{S}(f)\|_{L^p(\rn)} \leq C \|f\|_{L^p(\rn)},\end{equation}
 where, for all $x\in\rn$,
  $$\wt{S}( f)(x):=\lf\{\int_0^\infty \dashint_{B(0,t)} |f\ast \mathcal{K}_t(x+y)|^2\, dy \f {dt}{t}\r\}^{\f12}.$$
 In particular, if $n\leq4$, then \eqref{2-1-1'} holds true for all $p\in(1,\fz)$.

 Conversely, if \eqref{2-1-1'} holds  true for some $p\in(1,\fz)$ and all $f\in L^p(\rn)$, then $p\ge \f {2n}{4+n}$.
\end{lemma}

In order to show Lemma \ref{Lem'}, we need some technical lemmas.
Let $\varphi,\,\psi\in \cs(\rn)$ be two radial Schwartz functions such that $\supp\widehat{\varphi}
\subset \{ \xi\in\rn:\ \f 12 \leq |\xi| \leq 2\} $,
$ \sum_{j\in\mathbb{Z}} \widehat{\varphi}( 2^{-j} \xi) =1$ for all $\xi\neq 0$,
$\supp \widehat{\psi} \subset \{ \xi\in\rn:\ \f 14 \leq |\xi|\leq 4\}$ and $\widehat{\psi}(\xi) =1$
if $ \f 12 \leq |\xi|\leq 2$.
Clearly,
$ f\ast \varphi_{2^{-j}} = f\ast \varphi_{2^{-j}} \ast \psi_{2^{-j}}$.
Then we see that, for all $x\in\rn$,
 \begin{align*}
    &\wt{S}(f)(x)\\
    &\quad\sim \lf\{\int_{0}^{\infty} \int_{\rn} \lf|\sum_{j\in\mathbb{Z}} \sum_{k\in\mathbb{Z}}[ f\ast\varphi_{2^{-j-k}}\ast  \mathcal{K}_t (x+y)]\chi_{[2^{-k}, 2^{-k+1}]}(t)\r|^2 \,\chi_{B(0,t)}(y)\, dy \f {dt}{t^{n+1}}\r\}^{\f12},\end{align*}
   which,  together with the Minkowski inequality,  implies that, for all $x\in\rn$,
    \begin{align}\label{decom}
    \wt{S}(f)(x)&\ls\sum_{j\in\mathbb{Z}} \lf\{ \int_{0}^{\infty} \dashint_{B(0,t)} \lf| \sum_{k\in\mathbb{Z}}\lf[ f\ast\varphi_{2^{-j-k}}\ast  \mathcal{K}_t (x+y)\r]\chi_{[2^{-k}, 2^{-k+1}]}(t)\r|^2 \, dy \f {dt}t\r\}^{\f12}\\
    &\ls \sum_{j\in\mathbb{Z}} \mathcal{T}_j(f)(x),\notag\end{align}
    where, for all $j\in\mathbb{Z}$ and $x\in\rn$,
    \begin{align}\label{2-1'}
    \mathcal{T}_j (f)(x):
    &=\lf[\sum_{k\in\mathbb{Z}} \int_{2^{-k}}^{2^{-k+1}} \dashint_{B(0,2^{-k+1})} |f\ast \varphi_{2^{-j-k}}\ast \mathcal{K}_t(x+y)|^2 \, dy \f {dt}{t}\r]^{\f12}.
 \end{align}

We have the following lemma.
\begin{lemma}\label{lem-2-4'} Let $n\in[3,\fz)\cap\mathbb{N}$ and $j\in\mathbb{Z}$. Then there exists a positive constant $C$ such that, for all $f\in L^2(\rn)$,
$$\|\mathcal{T}_j(f)\|_{L^2(\rn)}\leq C2^{-2|j|} \|f\|_{L^2(\rn)}.$$
\end{lemma}

\begin{proof}
 By the Plancherel theorem and Lemma \ref{lem0'}(iv), we have
\begin{align*}
   \|\mathcal{T}_j(f)\|_{L^2(\rn)}^2&=\sum_{k\in\mathbb{Z}} \int_{2^{-k}}^{2^{-k+1}} \dashint_{B(0,2^{-k+1})} \int_{\rn}  |f\ast \varphi_{2^{-j-k}}\ast \mathcal{K}_t(x+y)|^2\, dx \, dy \f {dt}{t}\\
   &= \sum_{k\in\mathbb{Z}} \int_{2^{-k}}^{2^{-k+1}} \int_{\rn}  |f\ast \varphi_{2^{-j-k}}\ast \mathcal{K}_t(x)|^2\, dx \f {dt}{t}\\
   &\ls\sum_{k\in\mathbb{Z}} \int_{|\xi|\sim 2^{j+k}} |\widehat{f}(\xi)|^2   \int_{2^{-k}}^{2^{-k+1}}| \widehat{\mathcal{K}}(t\xi)|^2\, \f {dt}{t} d\xi\\
   &\ls\min\lf\{2^{4 j},\,2^{-4j}\r\} \|f\|_{L^2(\rn)}^2,
\end{align*}
where $|\xi|\sim2^{j+k}$ denotes the set $\{\xi\in\rn:\ 2^{j+k-1}<|\xi|\leq2^{j+k+1}\}$. This finishes the proof of Lemma \ref{lem-2-4'}.
\end{proof}

\begin{lemma}\label{lem-2-5'} Let $n\in[3,\fz)\cap\mathbb{N}$, $p\in(1,\fz)$ and
$\varepsilon\in (0,1)$. Then there exists a positive constant $C_{(\varepsilon,\,p)}$,
depending on $\varepsilon$ and $p$, such that, for all $j\in(-\fz,0]\cap\mathbb{Z}$
and $f\in L^p(\rn)$,
$$\|\mathcal{T}_j(f)\|_{L^p(\rn)}\leq C_{(\varepsilon,\,p)} 2^{4j(\f1q-\varepsilon)}\|f\|_{L^p(\rn)},$$
where $q:=\max\{p,p'\}$.
\end{lemma}

\begin{proof}
By Lemma \ref{lem-2-4'} and the Marcinkiewicz interpolation theorem
(see, for example, \cite[Theorem 1.3.2]{lg08}), it suffices to show that
\begin{equation}\label{se0}
\|\mathcal{T}_j(f)\|_{L^p(\rn)}\ls\|f\|_{L^p(\rn)},\   \   \ p\in(1,\fz).\end{equation}

By \eqref{2-1'} and \eqref{1-4'}, we see that, for all $x\in\rn$,
\begin{eqnarray}\label{te2}
    |\mathcal{T}_j (f)(x)|^2&&\ls \sum_{k\in\mathbb{Z}}
    \sup_{\substack {y\in B(0,2^{-k+1})\\
    t\in [2^{-k}, 2^{-k+1}]}} \lf|t^{-2}(B_t-I)(f\ast I_2\ast \varphi_{2^{-j-k}})(x+y)\r|^2\\
    &&\quad +\sum_{k\in\mathbb{Z}} \lf[M( f\ast \vi_{2^{-j-k}})(x)\r]^2,\noz
\end{eqnarray}
here and hereafter, $M$ denotes the \emph{Hardy-Littlewood maximal operator} defined by
 $$M(g)(x):=\sup_{B\subset \rn} \dashint_B |g(y)|\,dy, \ \ \ \ g\in L^1_\loc(\rn)
 \ \ {\rm and}\ \ x\in\rn,$$
where the supremum is taken over all balls $B$ in $\rn$ containing $x$.
Notice that, for all $\xi\in\rn$,
\begin{align*}
    t^{-2}\lf( (I-B_t)(f\ast I_2\ast \varphi_{2^{-j-k}})\r)^{\wedge} (\xi) =\widehat{\varphi}(2^{-j-k}\xi)\f {1-\widehat{\chi}(t\xi)} {t^2|\xi|^2} \widehat{f_{j+k}}(\xi)=:m_{j,k}(2^{-k}\xi) \widehat{f_{j+k}}(\xi),
\end{align*}
here and hereafter
$$m_{j,k}(\xi):=\widehat{\varphi}(2^{-j}\xi)\f {1-\widehat{\chi}(2^kt\xi)}{|2^kt\xi|^2},\ \ \ \ \ \xi\in\rn$$
 and $f_{j+k}:=f\ast\psi_{2^{-j-k}}$. Since $j\leq0$, by Lemma \ref{lem0'}(ii), for $t\in [2^{-k}, 2^{-k+1}]$, we have
\begin{equation*}
 |\nabla^i m_{j,k} (\xi)| \ls  2^{-ij}\chi_{[2^{j-1},2^{j+1}]}(|\xi|), \qquad i\in\mathbb{Z}_+,\ \xi\in\rn
\end{equation*}
and hence, for any $\ell\in\mathbb{Z}_+$,
$$|m_{j,k}^\vee(x)|\ls 2^{jn}(1+2^j|x|)^{-\ell},\   \  \  x\in\rn.$$
It follows that, for all $t\in [2^{-k}, 2^{-k+1}]$,
\begin{align}\label{2-2}
  & t^{-2}\sup_{|y|\leq 2^{-k+1}} |(I-B_t)(f\ast I_2\ast \varphi_{2^{-j-k}})  (x+y)|\\
  &\quad=\sup_{|y|\leq 2^{-k+1}}2^{kn}\lf|\int_{\rn} f_{j+k}(z) m_{j,k}^\vee (2^k(x+y-z))\, dz \r|\noz\\
   &\quad\ls 2^{(k+j)n}\int_{\rn} |f_{j+k}(z)| \sup_{|y|\leq  2^{-k+1}} \lf( 1+2^{k+j}|x+y-z|\r)^{-n-1}\, dz\noz\\
   &\quad\ls 2^{(k+j)n}\int_{\rn} |f_{j+k}(z)| \lf( 1+2^{k+j}|x-z|\r)^{-n-1}\, dz\ls M(f_{j+k})(x),\noz
\end{align}
where in the second inequality, we used the assumption $j\leq 0$.
By \eqref{te2}, \eqref{2-2}, the Fefferman-Stein vector-valued inequality (see \cite{fs}) and  the Littlewood-Paley characterization of $L^p(\rn)$ (see, for example, \cite[Theorem 6.2.7]{lg09}),
we obtain \eqref{se0} and hence complete the proof of Lemma \ref{lem-2-5'}.
\end{proof}
\begin{lemma}\label{lem-3-3}
Let $p\in(1,\fz)$. Then there exist positive constants $C$ and $C_{(p)}$
such that, for all $j\in\mathbb{Z}$, $f\in L^p(\rn)$ and $x\in\rn$,
\begin{align*}
|\mathcal{T}_j(f)(x)| &\leq C \lf(|\mathcal{T}_{j,1}(f)(x)| + |\mathcal{T}_{j,2}(f)(x)|\r),
\end{align*}
where
\begin{align}\label{te4}
\mathcal{T}_{j,1}(f)(x):=& \lf[\sum_{k\in\mathbb{Z}} 2^{4k} \dashint_{ B(x, 2^{-k+4})}\lf|f\ast I_2\ast \vi_{2^{-j-k}}(z)\r|^2\, dz\r]^{\f12}\end{align}
and $\mathcal{T}_{j,2}$ satisfies that
\begin{align}\label{te5}
\|\mathcal{T}_{j,2}(f)\|_{L^p(\rn)}\leq C_{(p)}\|f\|_{L^p(\rn)},\   \  p\in(1,\fz).
\end{align}
\end{lemma}

\begin{proof}
By \eqref{1-4'},  for all $t\in(0,\fz),\,g\in L_{{\rm loc}}^1(\rn)$ and  $x\in\rn$, we have
\begin{align*}
&\dashint_{B(0, 2t)}| g\ast \mathcal{K}_t(x+y)|^2\, dy\\
& \ls\dashint_{B(0, 2t)}\lf|\f{(B_t-I) (g\ast I_2)(x+y)}{t^2}\r|^2 \, dy  + \dashint_{B(0, 2t)}  |B_t (g)(x+y)|^2\, dy\\
&\ls t^{-4}\dashint_{B(x, 3t)}|g\ast I_2(z)|^2\, dz+ [M(g)(x)]^2,
\end{align*}
 which, together with \eqref{2-1'}, implies that
\begin{align*}
|\mathcal{T}_j (f)(x)| &\leq C \lf(\lf|\mathcal{T}_{j,1}(f)(x)\r| + \lf|\mathcal{T}_{j,2}(f)(x)\r|\r),
\ \ \ \ x\in\rn, \end{align*}
where $\mathcal{T}_{j,1}(f)$ is given by \eqref{te4} and
\begin{align*}
\mathcal{T}_{j,2}(f)(x):=& \lf\{\sum_{k\in\mathbb{Z}} \lf[M( f\ast \vi_{2^{-j-k}})(x)\r]^2\r\}^{\f12}.
\end{align*}
The inequality \eqref{te5} follows directly from the Fefferman-Stein vector-valued inequality (see \cite{fs}) and  the Littlewood-Paley characterization of $L^p(\rn)$ (see, for example, \cite[Theorem 6.2.7]{lg09}). This finishes the proof of Lemma \ref{lem-3-3}.
\end{proof}

For $J_{j,1}$, we have the following estimate.
\begin{lemma}\label{lem-3-4'}
Let $\va\in (0,1)$ and $p\in(1,2]$.
Then there exists a positive constant $C_{(\va,\,p)}$, depending on $\va$ and $p$,
such that, for all $j\in[0,\fz)\cap\mathbb{Z}$ and $f\in L^p(\rn)$,
\begin{equation}\label{3-10-0'}
\|\mathcal{T}_{j,1}(f)\|_{L^p(\rn)}\leq C_{(\va,\,p)} 2^{[ n(\f 1p-\f12)-2+\va]j}\|f\|_{L^p(\rn)}.
\end{equation}
\end{lemma}
\begin{proof}
By \eqref{te4} and the Plancherel theorem, we have
\begin{align}\label{se2'}
    \|\mathcal{T}_{j,1}(f)\|_{L^2(\rn)}^2 & \sim\sum_{k\in\mathbb{Z}} 2^{(n+4) k} \int_{B(0, 2^{-k+4})} \int_{\rn} |f\ast \vi_{2^{-j-k}}\ast I_2 (x+y)|^2 \, dx dy \\
    &\ls\sum_{k\in\mathbb{Z}} 2^{4 k} \int_{\rn} |f_{j+k} \ast \vi_{2^{-j-k}}\ast I_2(z)|^2 \, dz\noz\\
    &\ls \sum_{k\in\mathbb{Z}} 2^{4 k} \int_{|\xi|\sim 2^{j+k}} |\widehat{f_{j+k}}(\xi)|^2 |\xi|^{-4}\,d\xi\noz\\
    &\ls 2^{-4j} \sum_{k\in\mathbb{Z}} \|\widehat{f_{j+k}}\|_{L^2(\rn)}^2 \ls 2^{-4 j}\|f\|_{L^2(\rn)}^2\noz.
 \end{align}
This proves \eqref{3-10-0'} when $p=2$.

To show \eqref{3-10-0'} in the case when $p\in(1,2)$, by \eqref{se2'} and the Marcinkiewicz interpolation theorem (see, for example, \cite[Theorem 1.3.2]{lg08}), it suffices to prove that,
for all $\bz\in(0,1)$,
\begin{equation}\label{3-12'}
\lf|\{ x\in\rn:\ \mathcal{T}_{j,1}(f)(x)>\lz\}\r| \ls 2^{j(\f n2+\bz-2)}\f{\|f\|_{L^1(\rn)}}{\lz},\   \   \   \lz\in(0,\fz).
\end{equation}
To this end, let
$$\Sigma:=\lf (\rn\times \mathbb{Z}, \, \chi_{B(0, 2^{-k+4})}(y)\,dy\,d\gamma(k)\r),$$
where $d\gamma$ denotes the counting measure on $\mathbb{Z}$, namely, $\gamma(\{k\})=1$ for any $k\in\mathbb{Z}$.
Consider the following $L^2(\Sigma)$-valued operator:
 $$A_j (f)(x):= A_j (f)(x; y, k) :=  2^{(\f n2+2)k} f\ast \vi_{2^{-j-k}}\ast I_2(x+y), \   \ x\in\rn, \  y\in\rn, k\in \mathbb{Z}.$$
 Then, clearly,  $\mathcal{T}_{j,1}(f)(x)\sim \|A_j (f) (x)\|_{L^2(\Sigma)}$ and hence the proof of \eqref{3-12'} can be reduced to show that
 \begin{equation*}
    \lf|\{ x\in\rn:\ \|A_j(f)(x)\|_{L^2(\Sigma)}>\lz\}\r| \ls2^{j(\f n2+\bz-2)}\f{\|f\|_{L^1(\rn)}}{\lz},\   \   \  \ \  \lz\in(0,\fz)\ {\rm and}\ \bz\in (0,1).
\end{equation*}

Notice that
\begin{align*}
    2^{2 k}\lf( f\ast I_2\ast \vi_{2^{-j-k}}\r)^{\wedge} (\xi) =\f{\widehat{\vi}(2^{-j-k}\xi)} {(2^{-k}|\xi|)^2} \widehat{f}(\xi)=:2^{-2 j}m(2^{-j-k}\xi) \widehat{f}(\xi),
\end{align*}
where
$m(\xi):=\f{\widehat{\vi}(\xi)} {|\xi|^2}\in C_c^\infty(\rn)$. Thus,
$$ A_j(f)(x)=\int_{\rn} f(z) H_{j+k}(x+y-z)\, dz,\   \  (y,k)\in\Sigma,$$
where
$$H_{j+k}(z) :=2^{-2 j} 2^{\f n2 k} 2^{(j+k) n} m^\vee (2^{j+k} z),\ \ \ z\in\rn.$$
Therefore, by
$$\lf\|\|A_j(f)\|_{L^2(\Sigma)}\r\|_{L^2(\rn)}\ls 2^{-2 j}\|f\|_{L^2(\rn)}$$
 (due to \eqref{se2'}) and
the vector-valued Calder\'on-Zygmund theory (see \cite[Theorem 3.4]{G-CF}), it suffices to show that,
for any $\bz\in (0,1)$,
\begin{align}\label{3-13}
   \lf\|H_{j+k}(x+\cdot-z)-H_{j+k}(x+\cdot)\r\|_{L^2(\Sigma)} \ls 2^{j(\f n2+\bz-2)}\f {|z|^\bz}{|x|^{n+\bz}}\ \ \ \ {\rm for\ all}\ |x|>2|z|.
\end{align}

Now we prove \eqref{3-13}. Since $\vi$ is radial, we know that $m^\vee$ is also radial. Thus, there exists a radial Schwartz function $G$ on $\mathbb{R}$ such that $m^\vee(\xi)=G(|\xi|)$ for all $\xi\in\rn$. Define
$$ \widetilde{H}_{j+k} (u):= 2^{-2j} 2^{\f{nk}2} 2^{(j+k)n} G(2^{j+k}u),\   \   \  u\in\mathbb{R}.$$
Then $H_{j+k}(z)=\widetilde{H}_{j+k}(|z|)$.
Furthermore, since $G\in\cs(\rn)$, we know that,  for any given $\ell\in(0,\fz)$ and all $u\in \mathbb{R}$,
\begin{equation}\label{3-15}
    \lf| \widetilde{H}_{j+k} (u)\r|\ls 2^{-2j} 2^{\f {nk}2} 2^{(j+k)n} \lf(1+ 2^{j+k}|u|\r)^{-\ell}
\end{equation}and
\begin{equation}\label{3-14}
    \lf|\f {d}{du} \widetilde{H}_{j+k} (u)\r|\ls 2^{-2 j} 2^{\f {nk}2} 2^{(j+k)(n+1)}\lf (1+ 2^{j+k}|u|\r)^{-\ell},
\end{equation}
where the implicit positive constants depend on $\ell$, but are independent of $j,\,k$ and $u$.
For any $x,\,y,\,z\in\rn$, by the mean value theorem, there exists a real number $u$ between $|x+y-z|$ and $|x+y|$ such that
\begin{align*}
   &\lf|H_{j+k}(x+y-z)-H_{j+k}(x+y)\r|\\ &\quad=\lf|\widetilde{H}_{j+k}(|x+y-z|)-\widetilde{H}_{j+k}(|x+y|) \r|\\
   &\quad=\lf|\f d{du}\widetilde{H}_{j+k}(u)\r| \lf||x+y-z|-|x+y|\r|\leq |z|\lf|\f d{du}\widetilde{H}_{j+k}(u)\r|,
\end{align*}
which, combined \eqref{3-14}, implies that,  for any given $\ell\in(0,\fz)$ and all $x,\,y,\,z\in\rn$,
\begin{align}\label{3-16}
    &\lf|H_{j+k}(x+y-z)-H_{j+k}(x+y) \r|\\
     &\quad\ls 2^{-2 j} 2^{\f {nk}2} 2^{(j+k)(n+1)}|z|\lf[ \lf( 1+ 2^{j+k} |x+y|\r)^{-\ell}+\lf( 1+ 2^{j+k} |x+y-z|\r)^{-\ell}\r].\notag
\end{align}
On the other hand, using \eqref{3-15},
  we see that
 \begin{align}\label{3-17}
    &\lf|H_{j+k}(x+y-z)-H_{j+k}(x+y)\r |\\
     &\quad\ls 2^{-2j} 2^{\f {nk}2} 2^{(j+k)n}\lf[ \lf( 1+ 2^{j+k} |x+y|\r)^{-\ell}+\lf( 1+ 2^{j+k} |x+y-z|\r)^{-\ell}\r].\notag
\end{align}
Therefore, combining \eqref{3-16} with \eqref{3-17},
we find that, for all $\bz\in [0,1]$,  $\ell\in(1,\fz)$ and  $x,\,y,\,z\in \rn$,
 \begin{align}\label{3-18}
   & \lf|H_{j+k}(x+y-z)-H_{j+k}(x+y)\r|\\
   &\quad\ls 2^{-2 j} 2^{\f {nk}2} 2^{(j+k) n} 2^{(j+k) \bz} |z|^\bz
   \lf[ \lf( 1+ 2^{j+k} |x+y-z|\r)^{-\ell} + \lf(1+ 2^{j+k} |x+y|\r)^{-\ell}\r].\notag
 \end{align}

 Now we turn to show \eqref{3-13}. Assume that $|x|>2|z|$, and  write
 \begin{align*}
    &\lf\|H_{j+k}(x+\cdot-z)-H_{j+k}(x+\cdot)\r\|_{L^2(\Sigma)} \\
    &\quad=\lf\{ \sum_{k\in \mathbb{Z}} \int_{B(0, 2^{-k+4})} \lf|H_{j+k}(x+y-z)-H_{j+k}(x+y)\r|^2\, dy \r\}^{\f12}
    \leq J_1(x, z)+J_2(x, z),
 \end{align*}
 where
 \begin{align*}
   J_1(x, z):=& \lf\{ \sum_{2^{k-8} \ge  |x|^{-1}} \int_{B(0, 2^{-k+4})} \lf|H_{j+k}(x+y-z)-H_{j+k}(x+y)\r|^2\, dy \r\}^{\f12},\\
   J_2(x, z):=& \lf\{\sum_{2^{k-8} < |x|^{-1}} \int_{B(0, 2^{-k+4})} \lf|H_{j+k}(x+y-z)-H_{j+k}(x+y)\r|^2\, dy \r\}^{\f12}.
 \end{align*}

  It is relatively easier to estimate $J_1(x, z)$. Indeed, when $|x|\ge 2^{8-k} \ge 2^4 |y|$ and $|z|< \f 12 |x|$, we have
 $$|x+y|\sim |x+y-z|\sim |x|.$$
 From this and \eqref{3-16}, we deduce that, for any $\ell\ge n+\f32$,
 \begin{align*}
    \lf[J_1(x,z)\r]^2 &\ls 2^{-4 j}|z|^2 \sum_{2^k \ge 2^8 |x|^{-1}} 2^{2(j+k)(n+1)} \lf(2^{j+k}|x|\r)^{-2\ell}\\
    &\ls 2^{-4 j} |z|^2 2^{2j(n+1)} |x|^{-2(n+1)} 2^{-2j\ell},
 \end{align*}
 which further implies that, for all $i\in(1,\fz)$,
 $$ J_1(x,z) \ls2^{-ji} \f {|z|}{|x|^{n+1}},\ \ \ |x|>2|z|.$$

Finally, we estimate $J_2(x,z)$. Notice that, when $|x|\leq 2^{-k+8}$, $|z|< \f 12 |x|$ and $|y|\leq 2^{-k+4}$, one has $|x+y|\leq 2^{-k+9}$ and $|x+y-z|\leq 2^{-k+9}$.
 Therefore, from \eqref{3-18}, we deduce that, for any $\bz\in (0,1)$,
 \begin{align*}
 &\int_{B(0, 2^{-k+4})} \lf|H_{j+k}(x+y-z)-H_{j+k}(x+y)\r|^2\, dy \\
 &\quad\ls 2^{-4 j} 2^{nk} 2^{2(j+k)n} 2^{2(j+k)\bz} |z|^{2\bz} \int_{B(0, 2^{-k+9})} (1+2^{j+k}|\nu|)^{-n-1} \, d\nu\\
 &\quad\ls |z|^{2\bz} 2^{2k(n+\bz)}  2^{j(n+2\bz -4)},
\end{align*}
which implies that
\begin{align*}
    \lf[J_2(x,z)\r]^2 \ls2^{j(n+2\bz -4)} |z|^{2\bz} \sum_{2^k \leq 2^8 |x|^{-1}} 2^{2k(n+\bz)}
    \ls 2^{j(n+2\bz -4)} |z|^{2\bz}|x|^{-2(n+\bz)},
\end{align*}
and hence
$$J_2(x, z) \ls 2^{j(\f n2 +\bz-2)}|x|^{-(n+\bz)}|z|^\bz.$$
This, together with the estimate of $J_1(x,z)$, establishes the desired estimate \eqref{3-13}
and hence finishes the proof of Lemma \ref{lem-3-4'}.
\end{proof}

Now we prove Lemma \ref{Lem'}.
\begin{proof} [Proof of Lemma \ref{Lem'}]
We first prove \eqref{2-1-1'}. Notice that, when $p\in[2,\fz)$, \eqref{2-1-1'}
follows from Lemma \ref{Lem p>2}(ii), since
 $$\mathcal{S}\lf(f,\f{\lap f}{4+2n}\r)(x)\sim\widetilde{S}( \widetilde{f})(x)\sim\lf\{ \int_0^\infty \dashint_{B(0,t)} |\widetilde{f}\ast \mathcal{K}_t(x+y)|^2\, dy \f {dt}{t}\r\}^{\f12},$$
 where $\widetilde{f}:=-\lap f$. When $n\in\{1,2,3\}$ and $p\in(1,2)$,
 the desired conclusion is deduced from Theorem \ref{Thm0}.
 Thus, it suffices to prove \eqref{2-1-1'} in the case when $n\in[4,\fz)\cap\mathbb{N}$ and $p\in(\f{2n}{4+n},2)$.

By Lemmas \ref{lem-2-5'}, \ref{lem-3-3} and \ref{lem-3-4'}, we see that,
for any given $\va\in (0,1/2)$ and all $f\in L^q(\rn)$,
\begin{equation}\label{te6}
    \|\mathcal{T}_j(f)\|_{L^q(\rn)} \ls \lf\{1+2^{[n(\f 1q-\f12)-2+\va]|j|}\r\}\|f\|_{L^q(\rn)},\   \  j\in \mathbb{Z},\   \  q\in(1,2).
\end{equation}
For given $p\in(1,2)$, let $\va\in (0,1/2)$ be  small enough  such that
$\t=\f 2p -1+\f \va 2\in (0,1)$.  Let $q_\va$ be such that $\f 1p =\f \t {q_\va} +\f{1-\t}2$. Then
$ \f 1p -\f 12 =\t (\f 1{q_\va}-\f 12)$, which, in particular,  implies that $1<q_\va<2$ since $\t>2(\f 1p-\f12)$.  Thus, using \eqref{te6}, Lemma \ref{lem-2-4'} and the Marcinkiewicz interpolation theorem (see, for example, \cite[Theorem 1.3.2]{lg08}), we  conclude that, for all $j\in\mathbb{Z}$ and $f\in L^p(\rn)$,
\begin{align*}
    \|\mathcal{T}_j(f)\|_{L^p(\rn)} &\ls\lf\{2^{|j| \t [n(\f 1{q_\va}-\f12)-2+\f \va \t]} 2^{-2(1-\t)|j|}+ 2^{-2(1-\t)|j|}\r\}\|f\|_{L^p(\rn)}\\
    &\sim \lf\{2^{-|j|[ 2-n(\f 1p-\f12)]}+ 2^{-4|j|(1-\f 1p)}\r\}2^{|j|\va} \|f\|_{L^p(\rn)}.
\end{align*}
If $\f {2n}{4+n}<p<2$, then $n(\f 1p-\f12)<2$ and hence
 we may choose $\va\in (0,1/2)$ small enough in the above estimate such that
 $$\d:=\min\lf\{ 2-n\lf(\f 1p-\f12\r)-\va, 4\lf(1-\f1p\r)-\va\r\}>0.$$
 We then conclude that, for all $f\in L^p(\rn)$,
 $$\|\mathcal{T}_j(f)\|_{L^p(\rn)}\ls 2^{-|j|\d}\|f\|_{L^p(\rn)},\   \  j\in\mathbb{Z}.$$
The desired inequality \eqref{2-1-1'} then follows from \eqref{decom}.
This finishes the proof of \eqref{2-1-1'}.

Next, we show that, if \eqref{2-1-1'} holds true, then one must have  $p\ge \f {2n}{4+n}$. We only need to show the case when $n\in[5,\fz)\cap\mathbb{Z}$,
since we always assume $p\in(1,\fz)$. To this end, let $\phi$ be a radial Schwartz function on $\rn$ such that $$\chi_{[\f 12,2]}(|\xi|) \leq \widehat{\phi}(\xi)\leq \chi_{[\f 14, 4]}(|\xi|),\   \  \xi\in \rn$$
and $\phi_j(\cdot):=2^{jn}\phi(2^j\cdot)$, where $j$ is a sufficiently large positive integer. Then, for $x\in\rn$, we have
\begin{align}\label{4-1'}
\widetilde{S}(\phi_j)(x)
&\gtrsim \lf\{\int_1^2 \int_{B(0,1)} \lf| \f {B_t (\phi_j\ast I_2)(x+y) -(\phi_j\ast I_2)(x+y)}{t^2}\r.\r.\\
&\hs\hs\lf. +B_t(\phi_j)(x+y)\Bigg|^2\, dy dt\r\}^{\f12}\gtrsim J_{j,1}(x) -J_{j,2}(x)-J_{j,3}(x),\notag
\end{align}
where
\begin{align*}
J_{j,1}(x)&:=2^{-2}\lf[ \int_{B(0,1)} |(\phi_j\ast I_2)(x+y)|^2 \, dy \r]^{\f12},\\
J_{j,2}(x) &:=\lf[\int_1^2 \int_{B(0,1)} |B_t (\phi_j\ast I_2)(x+y)|^2 \, dy dt\r]^{\f12},\\
J_{j,3}(x)&:=\lf[ \int_1^2 \int_{B(0,1)} |B_t \phi_j(x+y)|^2 \, dy dt\r]^{\f12}.
\end{align*}

Notice that, for all $\xi\in\rn$,
\begin{align*}
    (\phi_j\ast I_2)^{\wedge}(\xi) =2^{-2 j} m(2^{-j}\xi),
\end{align*}
where
$m(\xi):=\f {\widehat{\phi}(\xi)}{|\xi|^2}$ is a nonnegative radial $C^\infty$ function
supported on $\{\xi\in\rn:\ \f 14\leq |\xi|\leq 4\}$ and satisfies that, for all $\xi\in\rn$,
$0\leq m(\xi)\leq 16$
and, for all $\f 12 \leq |\xi|\leq 2$,
\begin{align*}
\f 1{4} \leq m(\xi)=\f 1{|\xi|^2}\leq 4.
\end{align*}
We rewrite $ \phi_j\ast I_2(z) =2^{-2j} K_j(z),$
where $K_j(z):=2^{jn} m^\vee(2^jz)$ for all $z\in\rn$.
Then
\begin{align}\label{se8}
  \lf[J_{j,1}(x)\r]^2 =2^{-4 j-4} \int_{B(0,1)} |K_j(x+y)|^2\, dy.
\end{align}
Assume $|x|\leq \f 12$. Since $m$ is a Schwartz function, we know that, for any $\ell>n$,
\begin{align*}
    \int_{|y|\ge 1} |K_j(x+y)|^2\, dy &\ls 2^{2jn} \int_{|y|\ge 1} (1+ 2^j|y+x|)^{-4\ell}\, dy\\
    &\ls 2^{2jn} \int_{|y|\ge 1} (1+2^j|y|)^{-4\ell}\, dy
    \ls2^{-j\ell}\ls2^{-4j},
\end{align*}
which, together with \eqref{se8}, implies that, for any $|x|\leq \f 12$,
\begin{align*}
    \lf[J_{j,1}(x)\r]^2 &\gtrsim 2^{-4 j-4} \int_{\rn} |K_j(x+y)|^2\, dy -2^{-4j}
    \sim2^{-4 j-4}  \int_{\rn} |\widehat{K_j}(\xi)|^2\, d\xi-2^{-4 j}\\
    & \gtrsim2^{-4j-4} \int_{2^{j-1}\leq |\xi|\leq 2^{j+1}} |m(2^{-j}\xi)|^2\, d\xi-2^{-4 j}
    \gtrsim 2^{-4 j+jn},
\end{align*}
where, in the last inequality, we used the assumption that $j\in\mathbb{N}$ is sufficiently large and the definition of $m$.
This further implies that
\begin{equation}\label{4-6'}
   \|J_{j,1}\|_{L^p(\rn)} \gtrsim \lf(\int_{|x|\leq \f 12} |J_{j,1}(x)|^p\, dx \r)^{\f1p} \gtrsim 2^{-2 j+\f {jn}2}.
\end{equation}

For the term $J_{j,2}$, we see that, for all $x\in\rn$,
\begin{align*}
    J_{j,2}(x) &=\lf[\int_1^2 \int_{B(0,1)} \lf|\dashint_{B(x+y,t)} (\phi_j \ast I_2)(z)\, dz\r|^2 \, dy\, dt\r]^{\f12}\\
    &\ls\lf\{\int_1^2 \int_{B(0,1)} \lf[\dashint_{B(x,3)} |\phi_j \ast I_2(z)|\, dz\r]^2 \, dy\, dt\r\}^{\f12}\\
& \ls M(\phi_j\ast I_2)(x) =2^{-2 j} M(K_j)(x),
\end{align*}
which, together with the boundedness of $M$ on $L^p(\rn)$ with $p\in(1,\fz]$, further implies that
\begin{equation}\label{4-4'}
    \|J_{j,2}\|_{L^p(\rn)} \ls2^{-2 j}\|K_j\|_{L^p(\rn)}\ls 2^{-2 j} 2^{jn(1-\f1p)}.
\end{equation}

 For the third term $J_{j,3}$, we find that, for all $x\in\rn$,
\begin{align*}
    \lf[J_{j,3}(x)\r]^2
    & \ls\int_1^2 \int_{B(0,1)} \lf[\dashint_{B(x,3)} |\phi_j(z)|\, dz\r]^2 \, dy \, dt
    \ls [M\phi_j(x)]^2.
\end{align*}
From this and the boundedness of $M$ on $L^p(\rn)$ with $p\in(1,\fz]$, we deduce that
\begin{equation}\label{4-5}
    \|J_{j,3}\|_{L^p(\rn)}\ls\|\phi_j\|_{L^p(\rn)} \ls 2^{jn(1-\f 1p)}.
\end{equation}

Thus, combining the estimates \eqref{4-1'}, \eqref{4-6'}, \eqref{4-4'} and \eqref{4-5} with the assumption \eqref{2-1-1'}, we conclude that, for all sufficiently large $j\in\mathbb{N}$,
\begin{align*}
    2^{j(-2+\f n2)} &\ls \|J_{j,1}\|_{L^p(\rn)} \ls\|\widetilde{S}(\phi_j)\|_{L^p(\rn)} +
    \|J_{j,2} \|_{L^p(\rn)}+\|J_{j,3} \|_{L^p(\rn)}\\
    &\ls \|\phi_j\|_{L^p(\rn)} +2^{jn(1-\f 1p)}\ls 2^{jn(1-\f 1p)},
\end{align*}
which implies that
$ -2+\f n2 \leq n(1-\f 1p)$ and hence
$ p\ge \f {2n}{n+4}$.
This finishes the proof of Lemma \ref{Lem'}.
 \end{proof}
Now we are ready to show Theorem \ref{Thm2}
\begin{proof}[Proof of Theorem \ref{Thm2}]
We first prove (I). The equivalence between (ii) and (iii) of Theorem \ref{Thm2}(I) is from Lemma \ref{e34x}. This equivalence holds true for all $p\in(1,\fz)$ and $n\in\nn$.

To complete the proof, it suffices to prove the equivalence between (i) and (ii) of Theorem \ref{Thm2}(I).
Clearly, (i) $\Longrightarrow$ (ii) follows from Lemma
\ref{Lem'} with $n\in[4,\fz)\cap\mathbb{N}$ and $p\in(\f{2n}{4+n},2)$,
since, for all $x\in\rn$,
 \begin{eqnarray}\label{se7}
 \mathcal{S}\lf(f,\f{\lap f}{4+2n}\r)(x)=\widetilde{S}( \widetilde{f})(x)=\lf\{ \int_0^\infty \dashint_{B(0,t)} |\widetilde{f}\ast \mathcal{K}_t(x+y)|^2\, dy \f {dt}{t}\r\}^{\f12},
 \end{eqnarray}
 where $\widetilde{f}:=-\lap f$. Moreover, for all $f\in W^{2,p}(\rn)$,
\begin{equation}\label{(p,p)1'}
  \lf\|\CS\lf(f,\f{\lap f}{4+2n}\r)\r\|_{L^p(\rn)}\ls \|\lap  f\|_{L^p(\rn)}.
\end{equation}

Now we show (ii) $\Longrightarrow$ (i).
Assume $f,\,g\in L^p(\rn)$ such that $\CS(f,g)\in L^p(\rn)$. We shall prove that $g$
coincides with $\lap f$ modulo a positive constant. To this end,
take a non-negative radial $C^\fz$ function $\phi$ which is supported in $B(0,1)$
such that $\|\phi\|_{L^1(\rn)}=1$ and, for all $\ez\in (0,\fz)$ and $x\in\rn$,
let $\phi_\ez(x):=\ez^{-n}\phi\lf(x/\ez\r)$,
$f_\ez:=f\ast\phi_\ez$ and $g_\ez:=g\ast\phi_\ez$. Then $f_\ez\in W^{2,p}(\rn)$. Therefore,
by the above proved conclusion (i) $\Longrightarrow$ (ii),
we see that $\CS(f_\ez,\f{\lap f_\ez}{2n+4})\in L^p(\rn)$. By the Minkowski inequality, for all $x\in\rn$, we have
\begin{eqnarray}\label{se6}
\cs\lf(f_\ez,g_\ez\r)(x)&&=\lf\{\int_0^\fz \int_{B(x,t)}\lf|\lf[\frac{B_tf-f}{t^2}-B_tg\r]\ast\phi_\ez(y)\r|^2 \,dy\, \frac{dt}{t^{n+1}}\r\}^{\frac12}\\
   &&=\lf\{\int_0^\fz \int_{B(x,t)}\lf|\int_\rn\lf[\frac{B_tf(y-z)-f(y-z)}{t^2}\noz\r.\r.\r.\\
   &&\lf.\lf.\quad- B_tg(y-z)\Bigg]\phi_\ez(z)\,dz\r|^2 \,dy\, \frac{dt}{t^{n+1}}\r\}^{\frac12} \noz\\
   &&\le\int_\rn\lf\{\int_0^\fz \int_{B(x-z,t)}\lf|\frac{B_tf(y)-f(y)}{t^2}-B_tg(y)\r|^2 \,dy\, \frac{dt}{t^{n+1}}\r\}^{\frac12}\noz\\
   &&\quad\times\phi_\ez(z)\,dz \noz \\
   &&=\int_\rn \CS(f,g)(x-z)\phi_\ez(z)\,dz
   =\CS(f,g)\ast\phi_\ez(x).\noz
\end{eqnarray}
For all $x\in\rn$ and $\ez\in (0,\fz)$, define
\begin{eqnarray*}
  D_\ez(x) :=&& \lf\{\int_0^\fz \int_{B(x,t)}
  \lf|B_t g_\ez(y)-\frac{1}{2n+4}B_t\lf(\lap f_\ez\r)(y)\r|^2 \,dy\,\frac{dt}{t^{n+1}}\r\}^{\frac12} \\
   =&& |B(0,1)|^{1/2}\lf\{\int_0^\fz \dashint_{B(x,t)}\lf|\dashint_{B(y,t)}
   \lf[g_\ez(z)-\frac{1}{2n+4}\lap f_\ez(z)\r]\,dz\r|^2 \,dy\,\frac{dt}{t}\r\}^{\frac12}.
\end{eqnarray*}
Then
$$D_\ez(x)\ls \CS\lf(f_\ez,\f{\lap f_\ez}{2n+4}\r)(x)+\CS\lf(f_\ez,g_\ez\r)(x)\le\CS\lf(f_\ez,\f{\lap f_\ez}{2n+4}\r)(x)
+\CS\lf(f,g\r)\ast\phi_\ez(x),$$
 which implies that $D_\ez\in L^p(\rn)$ for $\ez\in (0,\fz)$ since $\CS(f_\ez,\f{\lap f_\ez}{2n+4})\in L^p(\rn)$ and
$\CS(f,g)\in L^p(\rn)$. Therefore, for all $\ez\in (0,\fz)$, $D_\ez(x)<\fz$ almost
everywhere. By the H\"{o}lder inequality, we know that
\begin{eqnarray*}
 E_\ez(x) :=&& |B(0,1)|^{1/2}\lf\{\int_0^\fz \lf|\dashint_{B(x,t)}\dashint_{B(y,t)} \lf[g_\ez(z)-\frac{1}{2n+4}\lap f_\ez(z)\r]\,dz\,dy\r|^2 \,\frac{dt}{t}\r\}^{\frac12}\\
   \le&& |B(0,1)|^{1/2}\lf\{\int_0^\fz\dashint_{B(x,t)}\lf|\dashint_{B(y,t)}
   \lf[g_\ez(z)-\frac{1}{2n+4}\lap f_\ez(z)\r]\,dz\r|^2\, dy\,\frac{dt}{t}\r\}^{\frac12}=D_\ez(x),
\end{eqnarray*}
which implies that $E_\ez(x)<\fz$ almost everywhere. From the Lebesgue differentiation theorem,
we deduce that, for all $\ez\in (0,\fz)$ and $x\in\rn$,
$$\lim_{t\to 0^+}\dashint_{B(x,t)}\dashint_{B(y,t)}
\lf[g_\ez(z)-\frac{1}{2n+4}\lap f_\ez(z)\r]\,dz\,dy=g_\ez(x)-\frac{1}{2n+4}\lap f_\ez(x),$$
where $t\to 0^+$ means $t>0$ and $t\to0$. Therefore, for almost every $x\in\rn$,
\begin{equation}\label{te7}
g_\ez(x)-\frac{1}{2n+4}\lap f_\ez(x)=0.
\end{equation}
By the continuity of $g_\ez$ and $\lap f_\ez$, for any $\ez\in(0,\fz)$, \eqref{te7}
holds true for every $x\in\rn$. Hence, $\frac 1{2n+4}\lap f_{\ez}\to g$ in $L^p(\rn)$ as $\ez\to0^+$.
Since $f_{\ez}\to f$ in $L^p(\rn)$ as $\ez\to0^+$, it follows that $\lap f_{\ez}\to\lap f$
as $\ez\to0^+$ in the sense of distribution. Thus, $\frac1{2n+4}\lap f=g$ almost everywhere
and hence $f\in W^{2,p}(\rn)$.
This proves (ii) $\Longrightarrow$ (i).

The proof of the inverse inequality of \eqref{(p,p)1'} is similar to that of the
inverse inequality of \cite[(2.14)]{hyy15}, the details being omitted. Thus, we complete the proof of Theorem \ref{Thm2}(I).

Now we prove  Theorem \ref{Thm2}(II). By \eqref{se7} and Lemma \ref{Lem'}, we know that, if (i) of Theorem \ref{Thm2}(I)
is equivalent to either (ii) or (iii), then we have
 $p\in[\f{2n}{4+n},\fz)$. This implies Theorem \ref{Thm2}(II) and hence finishes the proof of Theorem \ref{Thm2}.
\end{proof}

\begin{remark}\label{re1.1}
(i) It is still unclear whether the equivalence between (i) and either (ii) or (iii) of  Theorem \ref{Thm2}(I)
remains true in the endpoint case $p=\f{2n}{4+n}$ or not.

(ii) By Theorem \ref{Thm2}, we see that, in its item (I), (i)
is equivalent to (ii) for all $p\in(1,\fz)$ if and only if $n\in\{1,2,3,4\}$. Indeed, if $n\in\{1,2,3,4\}$,
then, from Theorems \ref{Thm0} and \ref{Thm2}, we deduce that (i) of Theorem \ref{Thm2}(I) is equivalent to (ii) of
Theorem \ref{Thm2}(I) for all $p\in(1,\fz)$. Conversely, if (i) is equivalent to (ii) for all $p\in(1,\fz)$, then, by Lemma \ref{Lem'}, one must have $n\in\{1,2,3,4\}$.
\end{remark}

\section{Proof of Theorem \ref{Thm1}\label{s2}}
\hskip\parindent
 For all $\al\in(0,2)$, $t\in(0,\fz)$ and $x\in\rn$, let   $\chi(x):=\f 1 {|B(0,1)|} \chi_{B(0,1)} (x)$,
 $\chi_t(x):=t^{-n} \chi(x/t)$ and \begin{align*}
    K(x):=\chi\ast I_\al(x) -I_\al(x),
 \end{align*}
 where  $I_\al$ is a distribution defined by $\widehat{I_\al}(\cdot):=|\cdot|^{-\al}$.  Let
 $K_t(x)=t^{-n} K(x/t)$ for all $x\in\rn$ and $t\in(0,\fz)$. Then
 \begin{equation}\label{1-4}\widehat{K}(t\xi) = \f{\widehat{\chi}(t\xi)-1}{t^\al|\xi|^{\al}}
 \ \ \ \ {\rm and}\ \ \ \
 f\ast K_t (x) = \f {(B_t -I)(f\ast I_\al)(x)}{t^\al}.\end{equation}

We have the following conclusions.

\begin{lemma}\label{lem0}
For all $\al\in(0,2)$, there exists a positive constant $C_{(\al)}$, depending on $\al$,
such that, for all $\xi\in\rn$,
 $$|\widehat{K}(\xi)|\leq C_{(\al)} \min\{ |\xi|^{2-\al}, |\xi|^{-\al}\}.$$
\end{lemma}
\begin{proof}
Since $\chi$ is radial, we know that $\widehat{\chi}$ is also radial. Hence, for all
$\xi\in\rn$,
\begin{eqnarray*}
\widehat{\chi}(\xi)=\widehat{\chi}(|\xi|e_1)=\dashint_{B(0,1)}e^{-ix_1|\xi|}\,dx,
\end{eqnarray*}
where $e_1:=(1,0,\ldots,0)\in\rn$.
Therefore, if $|\xi|\leq1$, then, by the Taylor expansion, we see that
\begin{eqnarray*}
\lf|\widehat{K}(\xi)\r|&&=\lf|\f{\widehat{\chi}(\xi)-1}{|\xi|^{\al}}\r|\\
&&=\lf|-\f{|\xi|^{2-\al}}2\dashint_{B(0,1)}x_1^2\,dx+
o(|\xi|^{2-\al})\dashint_{B(0,1)}x_1^4\,dx\r|
\ls|\xi|^{2-\al}.
\end{eqnarray*}
 If $|\xi|\ge1$, then, since $\widehat{\chi}$ is the Fourier transform of an integrable function,
  it follows that $\widehat{\chi}$ is a bounded function, which completes the proof of Lemma \ref{lem0}.
\end{proof}
\begin{lemma}\label{Lem}
Let $\al\in(0,2)$ and $p\in(\max\{1, \f{2n}{2\al+n}\},\fz)$. Then there exists a positive constant $C$
such that, for all $f\in L^p(\rn)$,
 \begin{equation}\label{2-1-1}\|\wt{S} (f)\|_{L^p(\rn)} \leq C \|f\|_{L^p(\rn)},\end{equation}
 where $$\wt{S}( f)(x):=\lf\{\int_0^\infty \dashint_{B(0,t)} |f\ast K_t(x+y)|^2\, dy \f {dt}{t}\r\}^{\f12}, \ \ \ x\in\rn.$$
 In particular, if $n\leq2\al$, then \eqref{2-1-1} holds true for all $p\in(1,\fz)$.

 Conversely, if \eqref{2-1-1} holds true for some $p\in(1,\fz)$ and all $f\in L^p(\rn)$, then $p\ge \f {2n}{2\al+n}$.
\end{lemma}

In order to show Lemma \ref{Lem}, we need some technical lemmas.
Let $\varphi,\,\psi\in \cs(\rn)$ be the same as in the proof of Lemma \ref{Lem'}.
Similar to \eqref{decom}, for all $x\in\rn$, we have
\begin{align}\label{decom'}
\wt{S}(f)(x)\ls\sum_{j\in\mathbb{Z}}T_j(f)(x),
\end{align}
where
\begin{align}\label{2-1}
T_j (f)(x):
&=\lf[\sum_{k\in\mathbb{Z}} \int_{2^{-k}}^{2^{-k+1}} \dashint_{B(0,2^{-k+1})} |f\ast \varphi_{2^{-j-k}}\ast K_t(x+y)|^2 \, dy \f {dt}{t}\r]^{\f12}.
\end{align}

By Lemma \ref{lem0}, we have the following lemma, whose proof is similar to that of Lemma \ref{lem0'},
the details being omitted.

\begin{lemma}\label{lem-2-4} For $\al\in(0,2)$,
there exists a positive constant $C$ such that, for all $j\in\mathbb{Z}$ and $f\in L^2(\rn)$,
$$\|T_j(f)\|_{L^2(\rn)}\leq C \min\lf\{2^{-\al j},\,2^{(2-\al)j}\r\} \|f\|_{L^2(\rn)}.$$
\end{lemma}

\begin{lemma}\label{lem-2-5} Let $\al\in(0,2)$, $\varepsilon\in (0,1)$ and $p\in(1,\fz)$.
Then there exists a positive constant $C_{(\varepsilon,\,\al,\,p)}$,
depending on $\varepsilon,\,\al$ and $p$,
such that, for all $j\in(-\fz,0]\cap\mathbb{Z}$ and $f\in L^p(\rn)$,
$$\|T_j(f)\|_{L^p(\rn)}\leq C_{(\varepsilon,\,\al,\,p)} 2^{(4-2\al)j(\f1q-\varepsilon)}\|f\|_{L^p(\rn)},$$
where $q:=\max\{p,p'\}$.
\end{lemma}

\begin{proof} Let $p\in(1,\fz)$.
By Lemma \ref{lem-2-4} and the Marcinkiewicz interpolation theorem
(see, for example, \cite[Theorem 1.3.2]{lg08}), it suffices to show that,
for all $j\in(-\fz,0]\cap\mathbb{Z}$ and $f\in L^p(\rn)$,
\begin{equation}\label{se0'}
\|T_j(f)\|_{L^p(\rn)}\ls \|f\|_{L^p(\rn)}.\end{equation}

By \eqref{2-1} and \eqref{1-4}, we see that, for all $x\in\rn$,
\begin{align}\label{te2'}
    |T_j (f)(x)|^2\ls \sum_{k\in\mathbb{Z}}
    \sup_{\substack {y\in B(0,2^{-k+1})\\
    t\in [2^{-k}, 2^{-k+1}]}} \lf|t^{-\al}(B_t-I)(f\ast I_\al\ast \varphi_{2^{-j-k}})(x+y)\r|^2.
\end{align}
Similar to \eqref{2-2}, we conclude that, for all $t\in [2^{-k}, 2^{-k+1}]$,
\begin{align*}
  & t^{-\al}\sup_{|y|\leq 2^{-k+1}} |(I-B_t)(f\ast I_\al\ast \varphi_{2^{-j-k}})  (x+y)|\ls M(f_{j+k})(x).
\end{align*}
From this, \eqref{te2'}, the Fefferman-Stein vector-valued inequality (see \cite{fs}) and  the Littlewood-Paley characterization of $L^p(\rn)$ (see, for example, \cite[Theorem 6.2.7]{lg09}),
we deduce that \eqref{se0'} holds true. This finishes the proof of Lemma \ref{lem-2-5}.
\end{proof}

\begin{lemma}\label{lem-3-4}
Let $\al\in(0,2)$, $\va\in (0,1)$ and $p\in(1,2]$.
Then there exists a positive constant $C_{(\va,\,\al,\,p)}$, depending on $\va,\,\al$ and $p$,
such that, for all $j\in[0,\fz)\cap\mathbb{Z}$ and $f\in L^p(\rn)$,
\begin{equation}\label{3-10-0}
\|T_{j}(f)\|_{L^p(\rn)}\leq C_{(\va,\,\al,\,p)}2^{[ n(\f 1p-\f12)-\al+\va]j}\|f\|_{L^p(\rn)}.\end{equation}
\end{lemma}

\begin{proof}
For all $t\in(0,\fz)$ and $x\in\rn$,
\begin{align*}
   &\dashint_{B(0, 2t)}| f\ast K_t(x+y)|^2\, dy\\
   & =\dashint_{B(0, 2t)}\lf|\f{(B_t-I) (f\ast I_\al)(x+y)}{t^\al}\r|^2 \, dy
   \ls t^{-2\al}\dashint_{B(x, 3t)}|f\ast I_\al(z)|^2\, dz.
\end{align*}
By this and \eqref{2-1}, we see that, for all $x\in\rn$
\begin{eqnarray}\label{se1}
T_j(f)(x)\ls\lf[\sum_{k\in\mathbb{Z}} 2^{2\al k} \dashint_{ B(x, 2^{-k+4})}\lf|f\ast I_\al\ast \vi_{2^{-j-k}}(z)\r|^2\, dz\r]^{\f12}=:T_{j,1}(f)(x).
\end{eqnarray}

On the other hand, similar to \eqref{3-12'}, for all $\bz\in(0,1)$, we conclude that
\begin{equation}\label{3-12}
    \lf|\{ x\in\rn:\ T_{j,1}(f)(x)>\lz\}\r| \ls  2^{j(\f n2+\bz-\al)}\f{\|f\|_{L^1(\rn)}}{\lz},\   \   \   \lz\in(0,\fz).
\end{equation}
 Thus, by \eqref{3-12} and \eqref{se1}, it is easy to see that
\begin{equation*}
    \lf|\{ x\in\rn:\ T_{j}(f)(x)>\lz\}\r| \ls  2^{j(\f n2+\bz-\al)}\f{\|f\|_{L^1(\rn)}}{\lz},\   \   \   \lz\in(0,\fz),
\end{equation*}
which combined with Lemma \ref{lem-2-4} with $j\in[0,\fz)\cap\mathbb{Z}$ and the Marcinkiewicz interpolation theorem (see, for example, \cite[Theorem 1.3.2]{lg08}), further implies \eqref{3-10-0}.
This finishes the proof of Lemma \ref{lem-3-4}.
 \end{proof}

Now we prove Lemma \ref{Lem}.
\begin{proof}[Proof of Lemma \ref{Lem}]
We first show \eqref{2-1-1}.
Notice that, when $p\in[2,\fz)$, \eqref{2-1-1} follows from Lemma \ref{Lem p>2}(i), since
 $$\mathcal{S}_\al(f)(x)\sim\widetilde{S}( \widetilde{f})(x)=\lf\{ \int_0^\infty \dashint_{B(0,t)} |\widetilde{f}\ast K_t(x+y)|^2\, dy \f {dt}{t}\r\}^{\f12},$$
 where $\widetilde{f}:=(-\lap)^{\al/2}f$. Thus, it suffices to show that for $p\in(\max\{1,\f{2n}{2\al+n}\},2)$, \eqref{2-1-1} remains true.

For $p\in(1,2)$, let $\va,\,\t,\,q_\va$ be as those in the proof of Lemma \ref{Lem'}.
If $j\in(0,\fz)\cap\mathbb{Z}$, then, by Lemmas \ref{lem-2-4}, \ref{lem-3-4} and the Marcinkiewicz interpolation theorem (see, for example, \cite[Theorem 1.3.2]{lg08}), we conclude that
\begin{align*}\|T_j(f)\|_{L^p(\rn)} &\ls2^{|j|  \t [ n(\f 1{q_\va}-\f12)-\al+\f \va \t]} 2^{-\al(1-\t)|j|}\|f\|_{L^p(\rn)}\\
    &\sim 2^{-|j|[\al-n(\f 1p-\f12)-\va]}\|f\|_{L^p(\rn)}.
\end{align*}
If $j\in(-\fz,0]\cap\mathbb{Z}$, then, by Lemmas \ref{lem-2-4} and \ref{lem-2-5}, and the Marcinkiewicz interpolation theorem (see, \cite[Theorem 1.3.2]{lg08}, for example), we also conclude that
\begin{align*}
\|T_j(f)\|_{L^p(\rn)} &\ls2^{|j|\t [-2(2-\al)(\f1{q_\va}-\va)]} 2^{-(2-\al)(1-\t)|j|}\|f\|_{L^p(\rn)}\\
    &\sim 2^{-|j|\{(2-\al)[(\f 2{q_\va}-1-2\va)\t+1]\}}\|f\|_{L^p(\rn)}.
\end{align*}
Notice that $\f {2n}{2\al+n}<p<2$ implies $n(\f 1p-\f12)<\al$, so
 we may choose $\va\in (0,1)$ small enough in the above estimate so that
 $$\d:=\min\lf\{ \al-n\lf(\f 1p-\f12\r)-\va, (2-\al)\lf[\lf(\f 2{q_\va}-1-2\va\r)\t+1\r]\r\}>0.$$
We then see that
 $$\|T_j(f)\|_{L^p(\rn)}\ls 2^{-|j|\d}\|f\|_{L^p(\rn)},\   \  j\in\mathbb{Z}.$$
From this and \eqref{decom'}, we deduce \eqref{2-1-1}.

Next, we show that, if \eqref{2-1-1} holds true for some $p\in(1,\fz)$ and all $f\in L^p(\rn)$, then one must have $p\ge \f {2n}{2\al+n}$. We only need to show the case when $n>2\al$, since we always assume $p\in(1,\fz)$.
 Let $\phi\in\cs(\rn)$ and $j\in\mathbb{N}$ be the same as in the proof of Lemma \ref{Lem'}. Then
\begin{align}\label{4-1}
    \widetilde{S}(\phi_j)(x) &\gtrsim \lf\{\int_1^2 \int_{B(0,1)} \lf| \f {B_t (\phi_j\ast I_\al)(x+y) -(\phi_j\ast I_\al)(x+y)}{t^\al} \r|^2\, dy dt\r\}^{\f12}\\
    &\gtrsim J_{j,1}(x) -J_{j,2}(x),\notag
\end{align}
where
\begin{align*}
J_{j,1}(x)&:=2^{-\al}\lf[\int_{B(0,1)} |(\phi_j\ast I_\al)(x+y)|^2 \, dy \r]^{\f12},\\
    J_{j,2}(x) &:=\lf[\int_1^2 \int_{B(0,1)} |B_t (\phi_j\ast I_\al)(x+y)|^2 \, dy dt\r]^{\f12}.
\end{align*}

Similar to \eqref{4-6'} and \eqref{4-4'}, we have
\begin{equation}\label{4-6}
   \|J_{j,1}\|_{L^p(\rn)} \gtrsim 2^{-\al j+\f {jn}2}\ \ \ {\rm and}\ \ \
   \|J_{j,2}\|_{L^p(\rn)} \ls 2^{-\al j} 2^{jn(1-\f1p)},
\end{equation}
respectively. Therefore, by \eqref{4-1}, \eqref{4-6} and the assumption \eqref{2-1-1}, we conclude that, for all sufficiently large $j\in\mathbb{N}$,
\begin{align*}
    2^{j(-\al+\f n2)} &\ls \|J_{j,1}\|_{L^p(\rn)} \ls\|\widetilde{S}(\phi_j)\|_{L^p(\rn)} +
    \|J_{j,2} \|_{L^p(\rn)}\\
    &\ls\|\phi_j\|_{L^p(\rn)} +2^{-\al j}2^{jn(1-\f 1p)}\ls 2^{jn(1-\f 1p)}.
\end{align*}
This implies that
$ -\al+\f n2 \leq n(1-\f 1p)$, and hence
$ p\ge \f {2n}{n+2\al}$,
which completes the proof of Lemma \ref{Lem}.
\end{proof}
Now we are ready to show Theorem \ref{Thm1}.

\begin{proof}[Proof of Theorem \ref{Thm1}]
The equivalence between (ii) and (iii) of Theorem \ref{Thm1}(I) is from Lemma \ref{e34x}. This equivalence holds true for all $p\in(1,\fz)$ and $n\in\nn$.

To complete the proof, it suffices to prove the equivalence between (i) and (ii) of Theorem \ref{Thm1}(I).
Clearly, (i) $\Longrightarrow$ (ii) follows from Lemma
\ref{Lem} when $p\in(\max\{1, \f{2n}{2\al+n}\},\fz)$,
since, for all $x\in\rn$,
 \begin{eqnarray}\label{te9}\mathcal{S}_\al(f)(x)\sim\widetilde{S}( \widetilde{f})(x)\sim\lf\{ \int_0^\infty \dashint_{B(0,t)} |\widetilde{f}\ast K_t(x+y)|^2\, dy \f {dt}{t}\r\}^{\f12},
 \end{eqnarray}
 where $\widetilde{f}:=(-\lap)^{\al/2}f$. Moreover, for all $f\in W^{\az,p}(\rn)$,
\begin{equation}\label{(p,p)1}
  \|\CS_\al(f)\|_{L^p(\rn)}\lesssim \|(-\lap)^{\al/2} f\|_{L^p(\rn)}.
\end{equation}

Now we show (ii) $\Longrightarrow$ (i).
Assume $f\in L^p(\rn)$ such that $\CS_\al(f)\in L^p(\rn)$.
Take a non-negative radial $C^\fz$ function $\phi$ which is supported in $B(0,1)$
such that $\|\phi\|_{L^1(\rn)}=1$ and, for all $\ez\in (0,\fz)$ and $x\in\rn$,
let $\phi_\ez(x):=\ez^{-n}\phi\lf(x/\ez\r)$.
Clearly, for $\ez\in (0,\fz)$, $f_\ez:=f\ast\phi_\ez\in W^{\al,p}(\rn)$. Thus,
by \cite[Lemma 2(i)]{AMV} and the above proved conclusion (i) $\Longrightarrow$ (ii),
we see that $\CS_\al\lf(f_\ez\r)\in L^p(\rn)$. Similar to \eqref{se6}, for all $x\in\rn$, we have
\begin{eqnarray*}
\cs_\al\lf(f_\ez\r)(x)\leq\CS_\al(f)\ast\phi_\ez(x),
\end{eqnarray*}
and hence $\{\mathcal{S}_\al(f_\epsilon)\}_{\epsilon\in(0,\fz)}$ is a bounded set of $L^p(\rn)$. Thus, there exist a function $h\in L^p(\rn)$ and a sequence $\epsilon_j\rightarrow0^+$ as $j\rightarrow\fz$ such that
$$(-\lap)^{\al/2}f_{\epsilon_j}\rightarrow h\ \ \ {\rm as}\ \ \ j\rightarrow\fz$$
in the weak-$\ast$ topology of $L^p(\rn)$. On the other hand, by \cite[Lemma 2(ii)]{AMV}, we know that $(-\lap)^{\al/2}f$ and  $(-\lap)^{\al/2}f_\epsilon$ are both tempered distributions and hence
$$(-\lap)^{\al/2}f_\epsilon\rightarrow(-\lap)^{\al/2}f \ \ \ {\rm as}\ \ \ \epsilon\rightarrow0$$
in the weak topology of distribution. Therefore $(-\lap)^{\al/2}f=h \in L^p(\rn)$,
which implies $f\in W^{\al,p}(\rn)$. This proves (ii) $\Longrightarrow$ (i).

Finally we prove the inverse inequality of \eqref{(p,p)1} by borrowing some ideas from \cite{dgyy2}.
Notice that
\begin{eqnarray}\label{Triebel2}
\|(-\lap)^{\al/2}f\|_{L^p(\rn)}&\sim& \|(-\lap)^{\al/2}f\|_{\dot F^0_{p,2}(\rn)} \sim \|f\|_{\dot F^\al_{p,2}(\rn)}\\
&\sim& \lf\|\lf\{\int_0^\fz \lf[\dashint_{B(\cdot,t)} \lf|\vi_t\ast f(y)\r|\,dy\r]^2
\,\frac{dt}{t^{2\al+1}}\r\}^{\frac12}\r\|_{L^p(\rn)},\noz
\end{eqnarray}
where $\dot F^0_{p,2}(\rn)$ and $\dot F^\al_{p,2}(\rn)$ denote
Triebel-Lizorkin spaces,  the second equivalence in
\eqref{Triebel2} is due to the well-known lifting property of Triebel-Lizorkin spaces,
and the third one follows from the Lusin area function characterization of Triebel-Lizorkin spaces (see, for example, \cite[Theorem 2.8]{u12} and its proof).
Here we take $\vi\in \cs(\rn)$ such that
$$\supp\widehat{\vi} \subset \lf\{\xi\in\rn:\ \ 2^{k_0-1}\le |\xi|\le 2^{k_0+1}\r\}$$
and $|\widehat{\vi}(\xi)|\ge {\rm constant}>0$ when $\frac352^{k_0}\le|\xi|\le \frac532^{k_0}$
for some $k_0\in\mathbb{Z}$ which will be determined later (It is well known that different $k_0$ gives equivalent quasi-norms of $\|f\|_{\dot F^\al_{p,2}(\rn)}$).
On the other hand, for all $\xi\in\rn$,
\begin{eqnarray*}
\lf(B_t f-f\r)^\wedge(\xi)=\lf[m(t\xi)-1\r] \widehat{f}(\xi)=:A(t|\xi|) \widehat{f}(\xi),
\end{eqnarray*}
where
$$m(\xi):= 1-2\gamma_n \int_0^1 \lf(1-u^2\r)^{\f
  {n-1}2} \lf(\sin \frac{u|\xi|}2\r)^{2} \, du,\quad \xi\in\rn,$$
$\gamma_n:=[\int_0^1(1-u^2) ^{\f {n-1}2} \, du]^{-1}$ and
$$A(s):=-2\gamma_n
\int_0^1 \lf(1-u^2\r)^{\f
  {n-1}2} \lf(\sin \frac{us}2\r)^{2} \, du, \quad s\in(0,\fz);$$
see \cite[Lemma 2.1]{dgyy2}.
Thus,
\begin{eqnarray*}
\vi_t\ast f=\lf[\widehat{\vi}(t\cdot) \widehat{f}(\cdot)\r]^\vee=
\lf[\frac{\widehat{\vi}(t\cdot)}{A(t|\cdot|)}A(t|\cdot|) \widehat{f}(\cdot)\r]^\vee=:\lf[\eta(t\cdot) A(t|\cdot|) \widehat{f}(\cdot)\r]^\vee,
\end{eqnarray*}
here $\eta(\xi):=\frac{\widehat{\vi}(\xi)}{A(|\xi|)}$ for all $\xi\in\rn$.
Since
$$\int_0^1 \lf(1-u^2\r)^{\f
  {n-1}2} \lf(\sin \frac{us}2\r)^{2} \, du\to 0,\quad s\to  0^+,$$
it follows that, when $s$ is small enough, then $|A(s)|>\frac 12$.
Thus, we can take $k_0$ small enough such that $\eta\in C_c^\fz(\rn)$, and hence $\eta^\vee$ is a Schwartz function. Then, for any $N\in\nn$ and $x\in\rn$, $|\eta^\vee(x)|\lesssim (1+|x|)^{-N}$, and  we find that, for all
$t\in(0,\fz)$ and $x\in\rn$,
\begin{eqnarray*}
  &&\dashint_{B(x,t)}\lf|\vi_t\ast f(y)\r|\,dy\\
 &&\quad=\dashint_{B(x,t)}\lf|\eta^\vee_t*(B_tf-f)(y)\r|\,dy\le \int_\rn\lf|\eta^\vee_t(\nu)\r|
   \dashint_{B(\nu-x,t)}\lf|(B_tf-f)(y)\r|\,dy\,d\nu\\
   &&\quad\ls \int_\rn\frac{t^{-n}}{\lf(1+|\nu+x|/t\r)^N}
   \dashint_{B(\nu,t)}\lf|(B_tf-f)(y)\r|\,dy\,d\nu\\
   &&\quad\sim\int_{|\nu+x|\le t}\frac{t^{-n}}{\lf(1+|\nu+x|/t\r)^N}
   \dashint_{B(\nu,t)}\lf|(B_tf-f)(y)\r|\,dy\,d\nu \\
   &&\quad\quad +\sum_{k=1}^\fz\int_{2^{k-1}t<|\nu+x|\le 2^kt}\frac{t^{-n}}{\lf(1+|\nu+x|/t\r)^N}
   \dashint_{B(\nu,t)}\lf|(B_tf-f)(y)\r|\,dy\,d\nu\\
   && \quad\ls\dashint_{|\nu+x|\le t}\dashint_{B(\nu,t)}\lf|(B_tf-f)(y)\r|\,dy\,d\nu \\
   &&\quad\quad +  \sum_{k=1}^\fz 2^{nk} 2^{-N(k-1)}\dashint_{|\nu+x|\le2^kt}
   \dashint_{B(\nu,t)}\lf|(B_tf-f)(y)\r|\,dy\,d\nu\\
   &&\quad\ls M\lf(\dashint_{B(\cdot,t)}\lf|(B_tf-f)(y)\r|\,dy\r)(-x)
   \lf[1+\sum_{k=1}^\fz2^{-(N-n)k}\r] \\
   &&\quad\sim M\lf(\dashint_{B(\cdot,t)}\lf|(B_tf-f)(y)\r|\,dy\r)(-x),
\end{eqnarray*}
where $M$ denotes the Hardy-Littlewood maximal operator and we took $N>n$.
Therefore, by \eqref{Triebel2}, the Fefferman-Stein vector-valued inequality (see \cite{fs})
and the H\"{o}lder inequality, we have
\begin{eqnarray*}
  \|(-\lap)^{\al/2} f\|_{L^p(\rn)} &\lesssim&\lf\|\lf\{\int_0^\fz \lf[ M\lf(\dashint_{B(\cdot,t)}\lf|(B_tf-f)(y)\r|\,dy\r)\r]^2
  \,\frac{dt}{t^{2\al+1}}\r\}^{\frac12}\r\|_{L^p(\rn)}  \\
   &\lesssim& \lf\|\lf\{\int_0^\fz \lf[\dashint_{B(\cdot,t)}\lf|(B_tf-f)(y)\r|\,dy\r]^2
  \,\frac{dt}{t^{2\al+1}}\r\}^{\frac12}\r\|_{L^p(\rn)} \\
   &\lesssim &  \lf\|\lf\{\int_0^\fz \int_{B(\cdot,t)}\lf|(B_tf-f)(y)\r|^2\,dy
  \frac{dt}{t^{n+2\al+1}}\r\}^{\frac12}\r\|_{L^p(\rn)}\sim\|\CS_\al(f)\|_{L^p(\rn)}.
\end{eqnarray*}
This finishes the proof of Theorem \ref{Thm1}(I).

Now we prove  Theorem \ref{Thm1}(II). By \eqref{te9} and Lemma \ref{Lem}, we know that the condition $p\in(\f{2n}{2\al+n},\fz)$ is nearly sharp, in the sense that, if Theorem \ref{Thm1}(i) is equivalent to either (ii) or (iii) of
Theorem \ref{Thm1}, then, we have
$$p\in \lf[\frac{2n}{2\al+n},\fz\r)\ {\rm for}\ n>2\al\ \quad
{\rm and}\quad \ p\in (1,\fz)\ {\rm for}\ n\leq2\al,$$
which implies Theorem \ref{Thm1}(II) and hence completes the proof of Theorem \ref{Thm1}.
\end{proof}

\begin{remark}\label{re1.2}
(i) When $n>2\az$, it is still unknown whether the equivalence between
(i) and either (ii) or (iii) of Theorem \ref{Thm1}(I) remains true in the endpoint case $p=\f{2n}{2\al+n}$ or not.

(ii) It is known that, when $p\in(0,\fz)$, $q\in(0,\fz]$ and $\az\in( n\max\{0,1/p-1/q\},1)$,
then $f\in F^\az_{p,q}(\rn)$ if and only if
$$\|f\|_{L^p(\rn)}+\lf\|\lf[\int_0^1 t^{-\az q}
\dashint_{B(\cdot,\,t)}\lf|f(\cdot)-
f(y)\r|^q\,dy\,\frac{dt}t\r]^{1/q}\r\|_{L^p(\rn)}<\fz,$$
which also serves as an equivalent quasi-norm of $F^\az_{p,q}(\rn)$, where $F^\al_{p,q}(\rn)$
 denotes the Triebel-Lizorkin space; see
\cite[Section 3.5.3]{t92}. When $q=2$ and $p\in(1,2)$, the restriction $\az\in( n\max\{0,1/p-1/q\},1)$
becomes $\az\in( n(1/p-1/2),1)$, which implies $p>\f{2n}{2\al+n}$. This indicates the lower boundary condition of $p$ in Theorem \ref{Thm1} is the same as that for the Lusin area function characterizations of Triebel-Lizorkin spaces via the first order differences.

(iii) We remark that (i) is equivalent to (ii) in Theorem \ref{Thm1}(I)
 for all $p\in(1,\fz)$ if and only if $n\leq2\al$. Indeed, if $n\leq2\al$,
then, from Theorem \ref{Thm1}, we deduce that (i) is equivalent to (ii)
for all $p\in(1,\fz)$. Conversely, if (i) is equivalent to (ii) in Theorem \ref{Thm1}(I)
for all $p\in(1,\fz)$, then, by Lemma \ref{Lem}, one must have $n\leq2\al$.
\end{remark}

\bigskip

\noindent Feng Dai

\medskip

\noindent  Department of Mathematical and Statistical Sciences,
University of Alberta, Edmonton, AB, T6G 2G1, Canada

\smallskip

\noindent {\it E-mail}: \texttt{fdai@ualberta.ca}
\bigskip

\noindent
\noindent Jun Liu, Dachun Yang (Corresponding author) and Wen Yuan

\medskip

\noindent School of Mathematical Sciences, Beijing Normal
University, Laboratory of Mathematics and Complex Systems, Ministry
of Education, Beijing 100875, People's Republic of China

\smallskip

\noindent{\it E-mails:} \texttt{junliu@mail.bnu.edu.cn} (J. Liu)

\hspace{1.1cm}\texttt{dcyang@bnu.edu.cn} (D. Yang)

\hspace{1.1cm}\texttt{wenyuan@bnu.edu.cn} (W. Yuan)

\end{document}